\documentclass[a4paper,10pt,reqno]{amsart}
\usepackage{amsmath,amsfonts}
\usepackage{algorithmic}
\usepackage{algorithm}
\usepackage{array}
\usepackage[caption=false,font=normalsize,labelfont=sf,textfont=sf]{subfig}
\usepackage{textcomp}
\usepackage{stfloats}
\usepackage{url}
\usepackage{verbatim}
\usepackage{graphicx}
\usepackage{cite}
\usepackage{adjustbox}
\usepackage{multirow}
\usepackage{amssymb}
\usepackage{pifont}
\newcommand{\cmark}{\ding{51}}%
\usepackage{multicol}
\usepackage[skip=0pt]{caption}
\usepackage{wrapfig}
\usepackage{tikz}
\usetikzlibrary{calc}

\DeclareMathOperator*{\argmin}{arg\,min}

\newtheorem{theorem}{Theorem}

\newtheorem{lemma}[theorem]{Lemma}
\newtheorem{remark}[theorem]{Remark}
\newtheorem{assumption}[theorem]{Assumption}
\newtheorem{definition}[theorem]{Definition}

\newcommand{\R}{\mathbb{R}}
\newcommand{\N}{\mathbb{N}}

\begin{document}

\title[Maximum Discrepancy Generative Regularization for SCSS]{Maximum Discrepancy Generative Regularization and Non-Negative Matrix Factorization for Single Channel Source Separation}

\author{Martin Ludvigsen and Markus Grasmair}
\address{Department of Mathematical Sciences, NTNU -- Norwegian University of Science and Technology, 7491 Trondheim, Norway}
\date{\today}

\begin{abstract}
  The idea of adversarial learning of regularization functionals
  has recently been introduced in the wider context of inverse problems.
  The intuition behind this method is the realization that
  it is not only necessary to learn the basic features that make up a class of signals
  one wants to represent, but also, or even more so, which features to avoid in the representation.
  In this paper, we will apply this approach to the training of generative models,
  leading to what we call Maximum Discrepancy Generative Regularization.
  In particular, we apply this to problem of source separation by means of Non-negative Matrix Factorization (NMF)
  and present a new method for the adversarial training of NMF bases.
  We show in numerical experiments, both for image and audio separation,
  that this leads to a clear improvement
  of the reconstructed signals, in particular in the case where little or
  no strong supervision data is available.
\end{abstract}


\maketitle

\section{Introduction}\label{se:intro}

Single Channel Source Separation (SCSS) is a type of inverse problem that is
concerned with the recovery of individual source signals from a measured mixed signal.
This problem arises in various real-world applications, such as speech
and music processing, biomedical signal analysis, and image processing. 
In such applications, the observed signal can be approximately modelled as a linear combination
of multiple sources, and the objective is to estimate the underlying sources from this mixture.
The simplest setting of this problem is that of denoising,
where one wants to separate a noisy signal into a clean signal and pure noise;
here the term ``noise'' is to be understood in a very wide sense and includes
all unwanted parts of the original signal.
The term ``single channel'' comes from the fact that we assume to only have
one measurement of the mixed signal, as opposed to the multichannel case where
we have several measurements with different weights for the component signals.

Mathematically, the problem can be formulated as follows:
Given a signal $v \in \R^m$, find $S$ individual signals
$u_i \in \mathbb{R}^m$ and potentially also weights $0 \le a_i\le 1$ with
$\sum_{i = 1}^S a_i = 1$ such that
\begin{equation}
  \sum_{i = 1}^S a_i u_i =: A\mathbf{u} = v.
  \label{eq:SCSS}
\end{equation}
Of course, this problem is vastly underdetermined, and it
cannot be reasonably solved without using some prior knowledge
about the properties of the component signals $u_i$ and, if necessary,
the weights $a_i$.
One possibility for encoding this prior knowledge is in terms of
smallness of some regularization's functionals $\mathcal{R}_i$ and $\mathcal{S}_i$.
With this assumption, one can then reconstruct the component signals $u_i$
by minimizing the (generalised, multi-parameter) Tikhonov functional
\[
  \min_{u_i,\, a_i} \Bigl[\frac{1}{2}\lVert A\mathbf{u}-v\rVert^2
  + \sum_{i=1}^S \lambda_i \mathcal{R}_i(u_i) + \sum_{i=1}^S \mu_i \mathcal{S}_i(a_i)\Bigr].
\]
Here $A\mathbf{u}$ is as defined in~\eqref{eq:SCSS}.
Moreover, the regularization parameters $\lambda_i$, $\mu_i > 0$ have to be chosen
in a suitable manner.

\subsection*{Training of regularization functionals}

A major challenge in this approach is
the definition of the different regularization functionals $\mathcal{R}_i$.
Within mathematical image processing, a classical approach
is to use hand-crafted regularization functionals that are intended to
capture the different properties of the component signals.
This idea has for instance been applied in~\cite{VesOsh03}
in order to separate an image into a cartoon component and a texture component.
Other examples of such decompositions have been proposed in~\cite{AujAubBlaCha05,BerVesSapOsh03},
see also \cite{GraNau23}.
In the recent years, alternative approaches have been developed
that instead use learned regularization \cite{mukherjee2023learned}.

Of particular interest for this paper is the idea
of adversarial regularization, see \cite{lunz2018adversarial}.
There the authors consider the more general inverse problem of solving
a linear equation of the form $Au = v$ with arbitrary linear operator $A$.
To that end, they assume that one has knowledge about the probability
distributions $\mathbb{P}_V$ and $\mathbb{P}_U$ of the measured data and the data of interest respectively.
In order to train a regularization functional, they use the probability
distribution $\mathbb{P}_V$ in order to define a new distribution of adversarial data on
the solution space by setting $\mathbb{P}_Z := (A^\dagger)_{\#} \mathbb{P}_U$.
Here $A^\dagger$ is the pseudo-inverse of the operator $A$,
and $(A^\dagger)_{\#}$ denotes the push-forward operation.

Then they define a regularization functional $\mathcal{R}$ for the solution
of the inverse problem by computing the \emph{Wasserstein distance}
$\mathbb{W}(\mathbb{P}_U,\mathbb{P}_Z)$ between the distributions
$\mathbb{P}_U$ and $\mathbb{P}_Z$.
More precisely, they set $\mathcal{R}$ to be the solution of
\begin{equation}
  \mathbb{W}(\mathbb{P}_U, \mathbb{P}_Z) =
  \min_{\mathcal{R}\colon\lVert \mathcal{R}\rVert_L \le 1} \mathbb{E}_{u \sim \mathbb{P}_U}[\mathcal{R}(u)]
  - \mathbb{E}_{u \sim \mathbb{P}_Z}[\mathcal{R}(u)].
  \label{eq:advregtrain}
\end{equation}
Here the minimisation is performed over all Lipschitz functions
$\mathcal{R} \colon U \to \mathbb{R}$ with Lipschitz constant $\lVert \mathcal{R} \rVert_L \le 1$.
Moreover $\mathbb{E}_{u\sim \mathbb{P}_X}$ denotes the expectation given the
distribution $\mathbb{P}_X$ of the variable $u$.
For the practical solution of the problem~\eqref{eq:advregtrain},
one restricts the admissible regularization functionals to
some a-priori defined class,
for instance that of neural networks with a given architecture. 

The intuition behind the adversarial approach is that $\mathcal{R}$ should yield low values for true solutions and
large values for naively inverted data, which can be treated as adversarial data, since
the measurement data $v$ can be assumed to include noise.

\subsection*{Training of generative models}

An alternative to this approach is the usage of generative models
for the component signals \cite{duff2021regularising,subakan2018generative,jayaram2020source}. 
Here we model data from source $i$ by a generating function
\[
  g_i \colon \R^{d_i} \to \R^m
\]
that maps the lower dimensional latent variables $h_i \in \R^{d_i}$ to a typical component
signal $u_i = g_i(h_i)$.
Then one can reconstruct the component signals $u_i$
by solving the Tikhonov functional
\begin{equation}\label{eq:TikGen}
  \min_{u_i,\, a_i,\, h_i} \Bigl[\frac{1}{2}\lVert A\mathbf{u}-v\rVert^2
  + \sum_{i=1}^S \bigl(\lambda_i \mathcal{R}_i(h_i) + \mu_i \mathcal{S}_i(a_i)\bigr)\Bigr]
  \ \text{ s.t. } g_i(h_i) = u_i,\ i=1,\ldots,S.
\end{equation}
Common choices for the regularization functionals
$\mathcal{R}_i \colon \R^{d_i} \to \R$ in this case are the squared Euclidean norm
$\mathcal{R}_i(h_i) = \frac{1}{2}\lVert h_i\rVert_2^2$ modelling
a Gaussian distribution of the latent variables,
or the $1$-norm $\mathcal{R}_i(h_i) = \lVert h_i \rVert_1$ modelling
a Laplace distribution of the latent variables and enforcing
a sparse representation of the component signal $u_i$.

The challenge in this approach is the training of
the generating functions $g_i$ for the different component signals.
The regularization terms, in contrast, are chosen a-priori
and are in fact incorporated in the training process of the models,
as are the regularization parameters.

In this paper, we will focus on this approach.
More precisely, we will introduce a novel method
for the training of the generating functions $g_i$ that is
based on the idea of adversarial regularization in~\cite{lunz2018adversarial}.

\subsection*{Other approaches}
As an alternative to generative regularization for inverse problems, it is possible to discriminatively train end-to-end mappings from observed data $v$ to the separated sources $\mathbf{u}$ \cite{jin2017deep}.
In recent years, this has been particularly common for speech source separation and enhancement \cite{vincent2018audio,wang2018supervised}.
Such approaches can be effective in situations with large amounts of high quality data, as they can be easily fine-tuned to specific data.
However, this makes them less generalizable compared to a generative approach, which can be used in a plug and play manner.

\subsection*{Overview and structure}

In Section~\ref{sec:adversarial} we will introduce our main ideas
for the adversarial training of generative models in a general setting.
In Section~\ref{sec:NMF} we will then consider the particular case
of Non-negative Matrix Factorization (NMF), which is a traditional method
in the context of source separation.
Details of the numerical implementation are described in Section~\ref{sec:numerics},
and in Section~\ref{sec:numerical} we present numerical experiments
both for image and audio data.
The appendices contain proofs of the theoretical results of
this paper as well as further details concerning the numerics.

\section{Adversarial training for SCSS models}
\label{sec:adversarial}

We will now propose a strategy for adapting the
idea of adversarial regularization to the training
of generative models for source separation.

\subsection{Notation and assumptions}

We assume that we are given training data consisting of
mixed signals $v^{(k)}$ as well as clean individual signals $u_i^{(k)}$
(see Section~\ref{ss:supervision} below for different possibilities
how these signals could be given).
Our goal is to use these training data to construct generating
functions $g_i$ such that we can use~\eqref{eq:TikGen} in order
to separate new mixed signals $v$.

For each of the sources $i$ we choose a class $\mathcal{G}_i$
of generating functions $g_i \colon \R^{d_i} \to \R^m$
with which we want to represent signals from this source,
as well as corresponding regularization functionals
$\mathcal{R}_i \colon \R^{d_i} \to \R_{\ge 0} \cup\{+\infty\}$
and
$\mathcal{S}_i \colon [0,1] \to \R_{\ge 0} \cup\{+\infty\}$
and regularization
parameters $\lambda_i > 0$ and $\mu_i > 0$.

As an example, $\mathcal{G}_i$ could be a class of neural
networks of a fixed architecture.
Another example, which we will discuss in more detail
in Section~\ref{sec:NMF} below, is the case of non-negative matrix factorization (NMF),
which is a traditional model in the context of audio processing.

\begin{assumption}\label{ass:1}
  We assume that the generating functions and regularization terms
  satisfy the following conditions:
  \begin{enumerate}
  \item The regularization terms $\mathcal{R}_i\colon \R^{d_i} \to \R_{\ge 0}\cup\{+\infty\}$
    are coercive, lower semi-continuous, and proper.
  \item The regularization terms $\mathcal{S}_i \colon [0,1] \to \R_{\ge 0}\cup\{+\infty\}$ are
    lower semi-continuous and proper.
  \item The evaluation mappings $\operatorname{eval} \colon \mathcal{G}_i \times \R^{d_i} \to \R^{m}$
    are continuous. 
  \end{enumerate}
\end{assumption}

Given a generator $g_i \in \mathcal{G}_i$ and a signal $u \in \R^m$,
we denote
\begin{equation}\label{eq:decoder}
  h_i^*(g_i,u;\lambda) := \argmin_{h \in \R^{d_i}} \Bigl[\lVert g_i(h) - u\rVert^2 + \lambda_i \mathcal{R}_i(h)\Bigr].
\end{equation}
That is, $h_i^*(g_i,u;\lambda_i)$ is an ``optimal'' parameter for representing the
signal $u$ by the generator $g_i$, given our choice of regularization parameter $\lambda_i$.
For ease of notation, we will assume $\lambda_i$ is fixed, and instead use the notation $h_i^\ast(g_i,u)$. 
Moreover, the mapping $h_i^*(g_i,\cdot) \colon \R^m \to \R^{d_i}$ is the
encoder corresponding to the generator (or decoder) $g_i$.
In many machine learning contexts, the mapping $h_i^*(g_i,u)$ is instead trained alongside $g_i$, leading
to what is commonly called an autoencoder.
Because of the continuity of $g_i$ and the lower semi-continuity and coercivity
of $\mathcal{R}_i$, such optimal parameters $h_i^*(g_i,u)$ exist for each $g_i$ and $u$.
Without additional assumptions, however, these need not be unique,
and thus we should interpret $h_i^*(g_i,u)$ as the collection of
all solutions of~\eqref{eq:decoder}.
In order to simplify notation,
we assume in the following that we select an arbitrary one of those solutions.

Next, we denote by $\pi_i \colon \mathcal{G}_i \times \R^m \to \R^m$ the mapping
\begin{equation}\label{eq:projection}
  \pi_i(g_i,u) := g_i\bigl(h_i^*(g_i,u)).
\end{equation}
That is, $\pi_i(g_i,\cdot)$ is the composition of the generator $g_i$
with the encoder $h_i^*(g_i,\cdot)$.
Informally, we interpret $\pi_i(g_i,u)$ as the ``projection'' of
the possibly noisy signal $u$ onto the manifold generated by the model $g_i$.
Because of the non-uniqueness issues discussed above,
this projection may depend on the choice of $h_i^*(g_i,u)$.
Figure~\ref{fig:notation} illustrates the relation between generator,
encoder, and the mapping $\pi$.

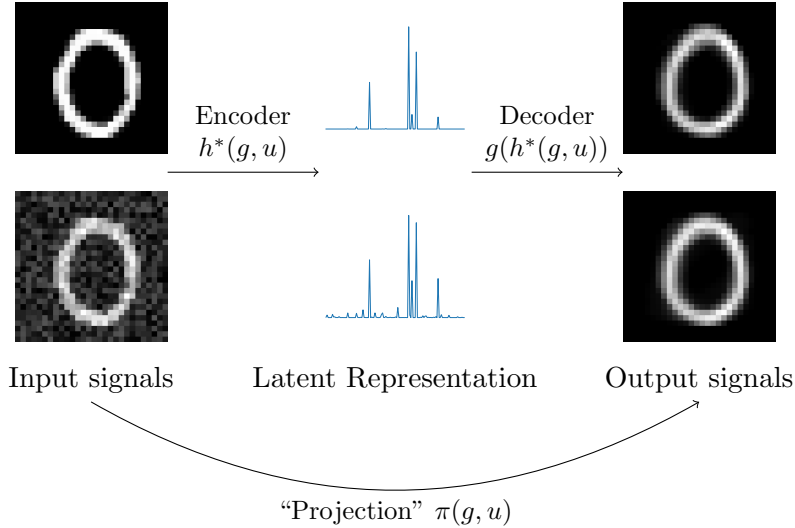
\begin{figure}[t]
  \begin{tikzpicture}
    \tikzstyle{block} = [inner sep=0pt, minimum size=0cm]
    
    \node[block] (img1) at (0,0) {\includegraphics[width=2cm]{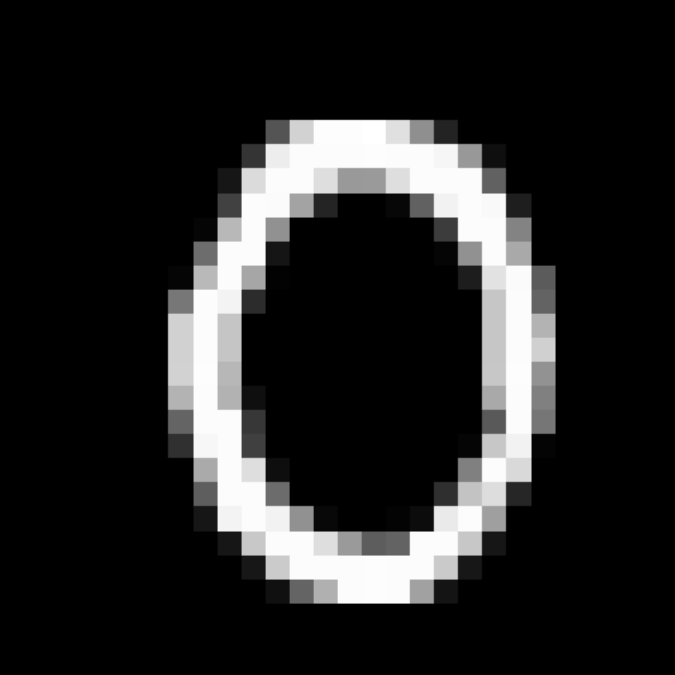}};
    \node[block] (img2) at (0,-2.5) {\includegraphics[width=2cm]{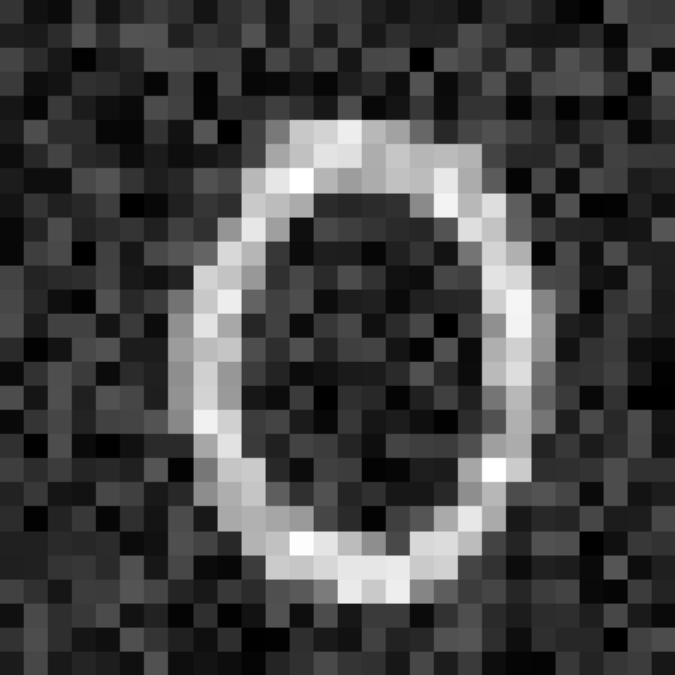}};
    
    \node[block] (img3) at (4,0) {\includegraphics[width=2cm]{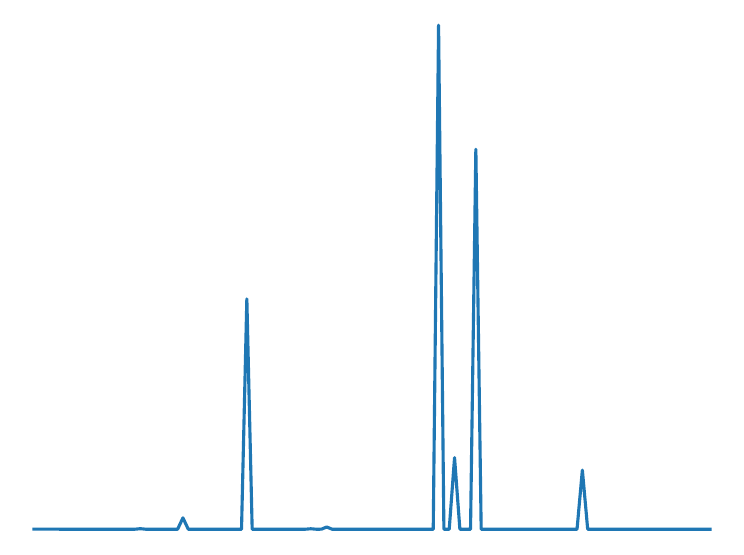}};
    \node[block] (img4) at (4,-2.5) {\includegraphics[width=2cm]{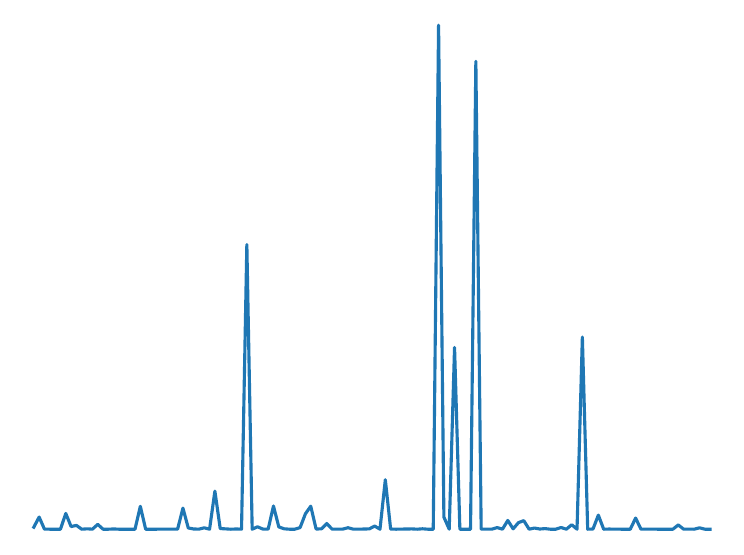}};
    
    \node[block] (img5) at (8,-0) {\includegraphics[width=2cm]{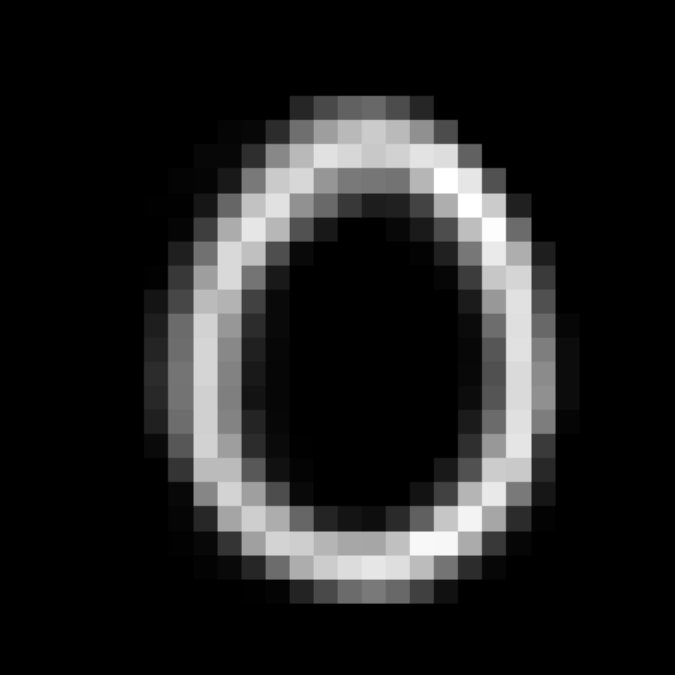}};
    \node[block] (img6) at (8,-2.5) {\includegraphics[width=2cm]{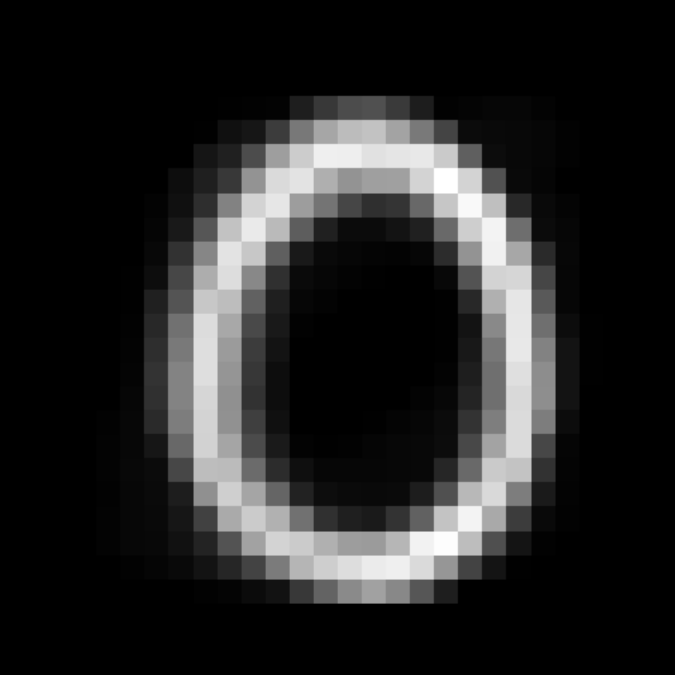}};
    
    \node (input) at (0, -4) {\large Input signals} ;
    
    \node (latent) at (4, -4) {\large Latent Representation};
    
    \node (output) at (8,-4) {\large Output signals};
    
    \draw[->] ($(img1.east)!0.5!(img2.east)$) -- ($(img3.west)!0.5!(img4.west)$) node[midway,above,align=center] {Encoder \\ $h^\ast(g,u)$};
    
    \draw[->] ($(img3.east)!0.5!(img4.east)$) -- ($(img5.west)!0.5!(img6.west)$) node[midway,above, align=center] {Decoder \\ $g(h^\ast(g,u))$};
    
    \draw[->] (input.south) to[out=-30,in=210] node[midway,below] {``Projection'' $\pi(g,u)$} (output.south);
    
    \end{tikzpicture}
    \caption{A general encoder-decoder framework. Here the encoding and decoding is done using NMF with sparsity
      (see Section~\ref{sec:NMF} for details).
    While the encoding and decoding process does not perfectly reconstruct signals, especially when using low complexity models, they are still useful as a prior on signals.
    They are also robust against noise, as we see that the process yields similar outputs for a signal and a noisy version of that signal.}
    \label{fig:notation}
\end{figure}

Finally, given generators $\mathbf{g} := (g_1,\ldots,g_S)$ and a mixed signal $v \in \R^m$
we denote by
\begin{equation}\label{eq:decomposition}
  \mathbf{u}^*(\mathbf{g};v) \in \R^{m\times S}
\end{equation}
the $u$-component of the solution of~\eqref{eq:TikGen}.
That is, the columns $u_i^*$ of $\mathbf{u}^*$ represent
an optimal decomposition of the signal $v$ given the
models $g_i$ for the different sources.
Again, this decomposition need in general not be unique
and we use the notation $\mathbf{u}^* = (u_1^*,\ldots,u_S^*)$
to denote any of the minimizers.

\begin{remark}
In our work, the term ``generative model'' refers to something broader than what is common in standard 
machine learning contexts. Generative models like GANs  are trained for the purpose of generating new data like $g(h) \approx u \sim \mathbb{P}_U$, where $h$ is sampled from a suitable probability distribution.
We here consider any mapping
$g$ so that the ``distance'' $\|\pi(g,u) - u\|$ is small for relevant data $u \sim \mathbb{P}_U$, and preferably large for irrelevant data.
Thus, the data has a generating function $g$. 
We do not require $g$ to be able to  generate data on its own, but instead that the data can be represented by $g$.
In particular, the generative model can be dictionary methods and other dimensionality reduction methods.
\end{remark}

\subsection{Strong and weak supervision}\label{ss:supervision}

For the training of the generative models $g_i$, it is necessary
to have access to some type of training data for the different sources.
In the context of source separation, the data will usually be available
in one of two forms, which we term ``strongly supervised'' and ``weakly supervised'' data,
see Figure~\ref{fig:supervision} for an illustration.

\begin{wrapfigure}{r}{0.5\textwidth}
  \begin{center}
  \begin{tikzpicture}[scale=0.5,every node/.style={scale=0.5}]

    \foreach \i in {5,4,...,1} {
        \node[inner sep=0,anchor=south west, xshift=\i*1mm, yshift=\i*1mm] (image\i) at (0,0) {
            \includegraphics[width=2cm]{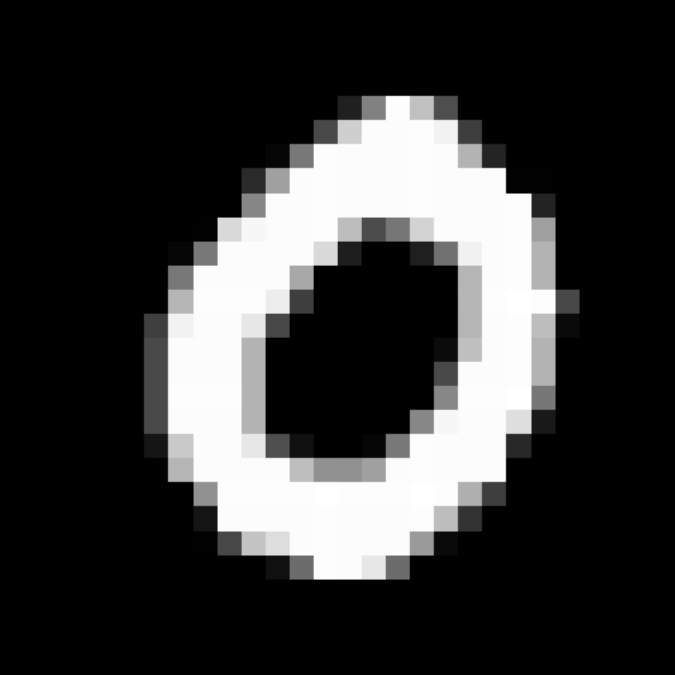}
        };
        \draw[xshift=\i*1mm, yshift=\i*1mm, thin, white] (image\i.south west) rectangle (image\i.north east);
    }

    \foreach \i in {5,4,...,1} {
        \node[inner sep=0,anchor=south west, xshift=\i*1mm, yshift=\i*1mm] (image\i) at (4,0) {
            \includegraphics[width=2cm]{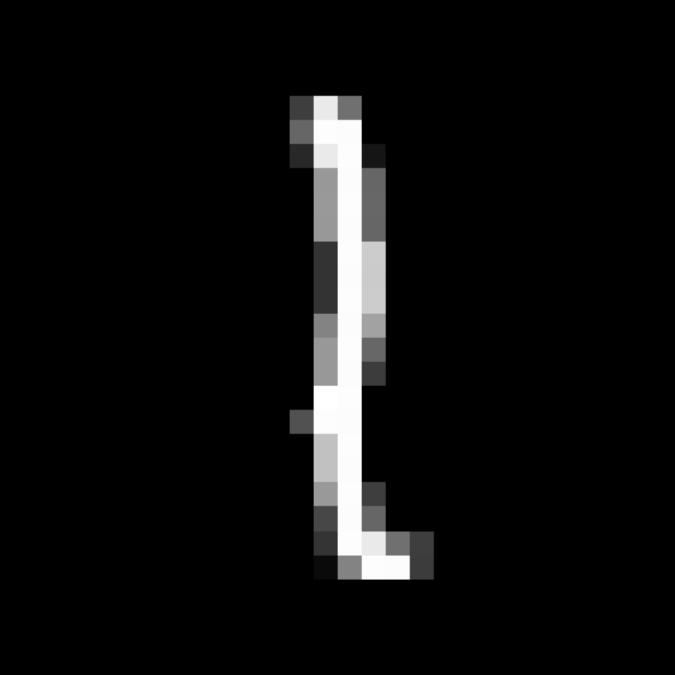}
        };
        \draw[xshift=\i*1mm, yshift=\i*1mm, thin, white] (image\i.south west) rectangle (image\i.north east);
    }

    \foreach \i in {5,4,...,1} {
        \node[inner sep=0,anchor=south west, xshift=\i*1mm, yshift=\i*1mm] (image\i) at (8,0) {
            \includegraphics[width=2cm]{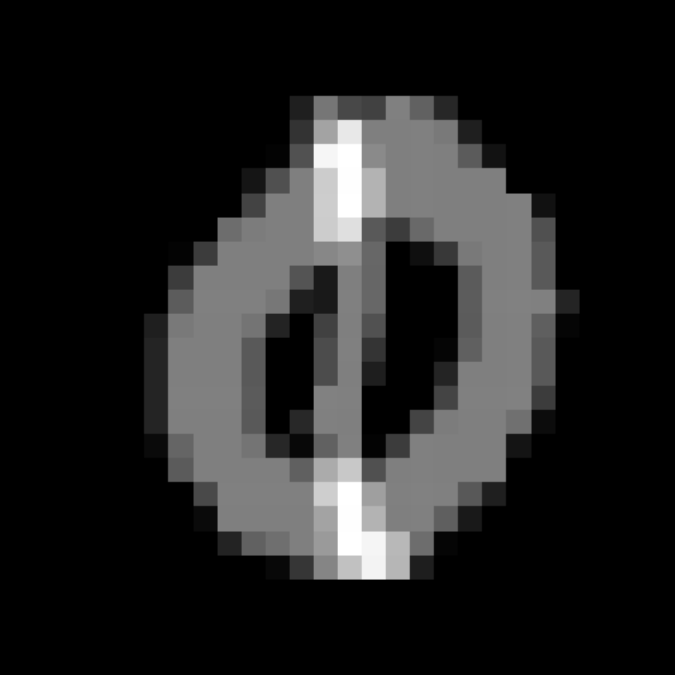}
        };
        \draw[xshift=\i*1mm, yshift=\i*1mm, thin, white] (image\i.south west) rectangle (image\i.north east);
    }

    \node at (3.33,1.33) {\huge +};  
    \node at (7.33,1.33) {\huge =};  
    \path let \p1=(image5.north east) in [/utils/exec={\xdef\MaxX{\x1}\xdef\MaxY{\y1}}];

    \draw[very thick] ([xshift=-1mm,yshift=-1mm]current bounding box.south west) rectangle ([xshift=1mm,yshift=10mm]\MaxX,\MaxY);

    \node at ([yshift=5mm]\MaxX/2,\MaxY) {\huge $\mathbb{P}_{U \times V}$};

    \node[above] at (current bounding box.north) {\huge Strong Supervision};
\end{tikzpicture}

\vspace{2mm}

\begin{tikzpicture}[scale=0.5,every node/.style={scale=0.5}]

    \foreach \i in {5,4,...,1} {
        \node[inner sep=0,anchor=south west, xshift=\i*1mm, yshift=\i*1mm, name = image1\i] at (0,0) {
            \includegraphics[width=2cm]{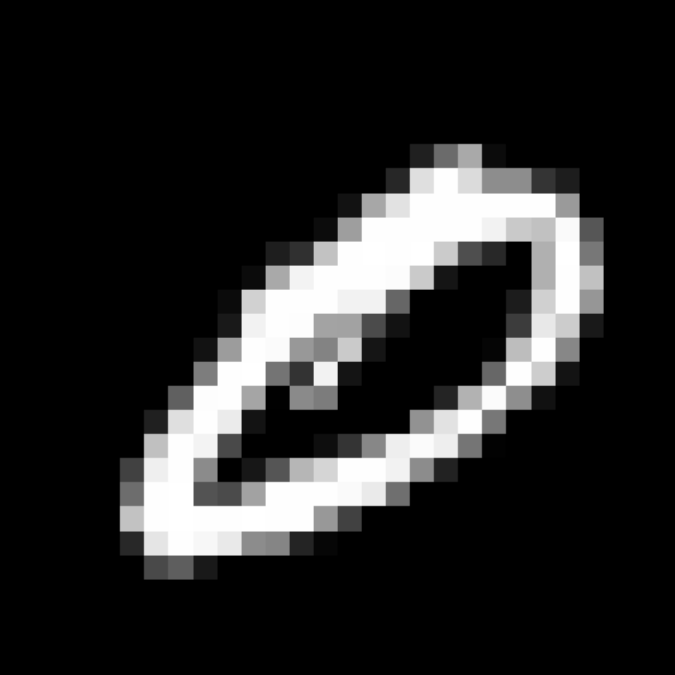}
        };
        \draw[xshift=\i*1mm, yshift=\i*1mm, thin, white] (image1\i.south west) rectangle (image1\i.north east);
    }
    \draw[very thick] ([xshift=-0mm,yshift=0mm]0,0) rectangle ([xshift=6mm,yshift=16mm]2cm,2cm);
    \node[above] at (1.33cm,2cm+6mm) {\Large $\mathbb{P}_{U_0}$};

    \foreach \i in {5,4,...,1} {
        \node[inner sep=0,anchor=south west, xshift=\i*1mm, yshift=\i*1mm, name=image2\i] at (4,0) {
            \includegraphics[width=2cm]{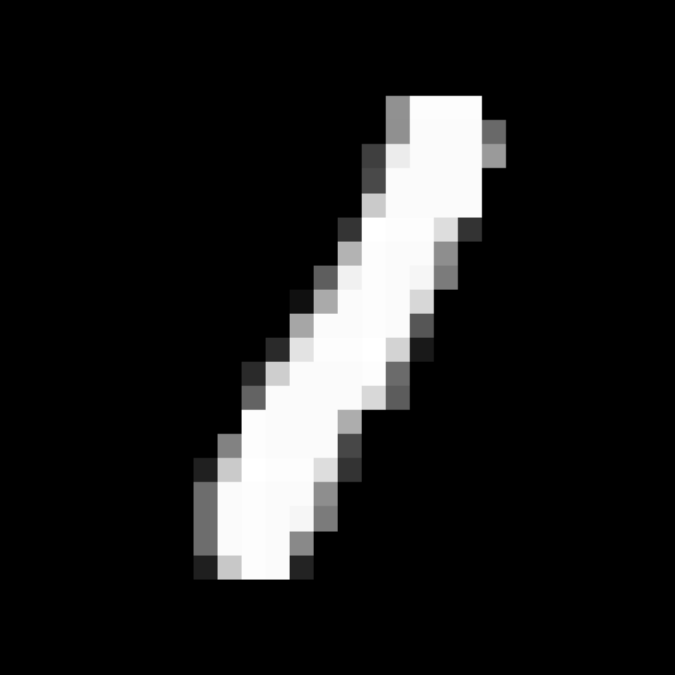}
        };
        \draw[xshift=\i*1mm, yshift=\i*1mm, thin, white] (image2\i.south west) rectangle (image2\i.north east);
    }
    \draw[very thick] ([xshift=-0mm,yshift=0mm]4,0) rectangle ([xshift=6mm,yshift=16mm]6cm,2cm);
    \node[above] at (5.33cm,2cm+6mm) {\Large $\mathbb{P}_{U_1}$};

    \foreach \i in {5,4,...,1} {
        \node[inner sep=0,anchor=south west, xshift=\i*1mm, yshift=\i*1mm, name=image3\i] at (8,0) {
            \includegraphics[width=2cm]{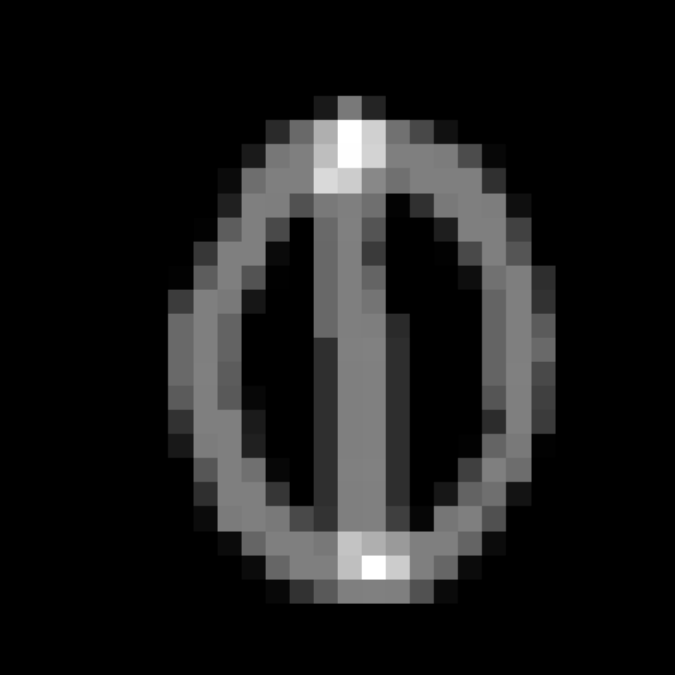}
        };
        \draw[xshift=\i*1mm, yshift=\i*1mm, thin, white] (image3\i.south west) rectangle (image3\i.north east);
    }
    \draw[very thick] ([xshift=-0mm,yshift=0mm]8,0) rectangle ([xshift=6mm,yshift=16mm]10cm,2cm);
    \node[above] at (9.33cm,2cm+6mm) {\Large$ \mathbb{P}_V$};

    \node[above] at (current bounding box.north) {\huge Weak Supervision};
\end{tikzpicture}
\end{center}
\caption{Illustration of the difference between strong supervision and weak supervision. In the strong supervision situation, all paired data is available.
In weak supervision, only data from the marginals are available, and these data are unpaired.}
\label{fig:supervision}
\end{wrapfigure}
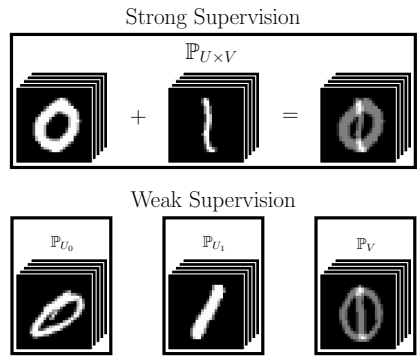

\subsubsection*{Strongly supervised data}
This is also called \textit{paired} data.
These are samples $v^{(k)}$ of the mixed data
together with the correctly separated sources 
such that $v^{(k)} = \sum_i a_i u_i^{(k)}$. Here the superscript refers to the $k$-th data in the available dataset.
In terms of the inverse problem~\eqref{eq:SCSS}, this means
that we have samples of the right hand side $v$ together with the
true solution $\mathbf{u}$ as well as the true operator $A$.
In a probabilistic setting, we can also interpret this as having
access to samples from the \emph{joint} probability distribution
$\mathbb{P}_{\mathcal{A} \times \mathcal{U} \times V}$ of weights, sources, and mixed sources.

If a sufficiently large amount of such training data is available,
then it can be possible to train dedicated discriminative models for a specific problem.
Compared to the method we are proposing, such discriminative models
can often yield superior results. However, in practice it is
usually unrealistic to obtain a sufficient amount of strongly supervised training data.

\subsubsection*{Weakly supervised data}
This is also called \emph{semi-supervised} or \textit{unpaired} data.
There exist different interpretations of the notion of weak supervision, but
in this paper we will focus on the setting where we
have access to samples of the different sources $u_i^{(\ell)}$
and also samples $v^{(k)}$ of mixed data,
but no information about the correct separation 
$v^{(k)} = \sum_i a_i^{(k)} u_i^{(k)}$, though we might have some probabilistic model for $a_i^{(k)}$.
In many cases, one can relatively easily obtain
clean signals from the different sources as well as clean mixed signals,
whereas it is much more difficult to obtain a mixed signal together
with the correct separation into its component signals.
Thus this type of data will in practice be much more common
than strongly supervised data.

Translated to the probabilistic setting, 
we have samples of the \emph{marginal} distributions
$\mathbb{P}_{U_i}$ for the source $i$ and $\mathbb{P}_V$ for the mixed data, and in addition
some model for the distribution $\mathbb{P}_{\mathcal{A}}$ of the weights.
However, we generally cannot assume that the distributions
are independent, as this would essentially mean that
each sample from any source can be arbitrarily mixed
with any sample from a different source.
Thus, from these weakly supervised data
we can infer only little or no information about
the joint distribution $\mathbb{P}_{\mathcal{A} \times \mathcal{U} \times V}$.

In practical applications, the amount of training
data for the different sources may vary significantly.
In the extreme case, it can be even possible that
we have no data at all from one of the sources.
This is particularly relevant in denoising applications,
where we may have access to noisy and clean signals, but no access to samples of ``clean noise''.

\begin{remark}
  If we have access to weakly supervised data and can
  reasonably assume that the distributions of the individual sources
  are independent, we can generate synthetic samples of supervised
  data by using the forward model~\eqref{eq:SCSS}, provided we have information about $\mathbb{P}_{\mathcal{A}}$.
  This allows us to synthetically simulate strongly supervised data.
  This is in particular common for audio tasks, as noise can often be assumed to be independent from the audio of interest.
\end{remark}

\begin{remark}
  In this paper we will not discuss the problem of
  \emph{blind source separation}, where one only has access
  to samples of mixed data,
  but not to samples from the unmixed sources.
  This setting requires knowledge-based assumptions about the
  unmixed sources in contrast to the data-driven methods to be
  discussed here.
  For an example of a possible approach
  that is based on non-negative matrix factorization, we refer to \cite{leplat2020blind}.
\end{remark}

\subsection{Adversarial training}\label{ss:adversarial_training}

For the training of the generators $g_i$, we now assume that
we have access to weak supervision data in the form of
samples of probability distributions $\mathbb{P}_{U_i}$, $i=1,\ldots,S$,
of the clean sources, samples of the distribution $\mathbb{P}_{V}$ of the mixed signals,
and a model $\mathbb{P}_{\mathcal{A}}$ of the distribution of the weights.
In addition, we might have access to a limited amount of strong supervision
data in the form of samples of the joint distribution
$\mathbb{P}_{\mathcal{A}\times \mathcal{U}\times V}$.
Following the idea of adversarial regularization introduced in~\cite{lunz2018adversarial}
we now propose to train the generators $g_i \in \mathcal{G}_i$ by
minimizing functionals of the form
\begin{equation}
  \mathcal{F}_i(\mathbf{g})
  = \tau_W \mathcal{F}^W_i(g_i) - \tau_A \mathcal{F}^A_i(g_i) + \tau_S \mathcal{F}^S_i(\mathbf{g}).
  \label{eq:full_model}
\end{equation}
Here $\mathcal{F}^W_i$, $\mathcal{F}^A_i$, and $\mathcal{F}^S_i$ are 
cost functions for weak supervision data, adversarial data (to be discussed below),
and strong supervision data, respectively;
$\tau_W$, $\tau_A$, $\tau_S \ge 0$ are regularization parameters that balance the importance
of the weak supervision, adversarial and strong supervision terms;
and $\mathbf{g} = (g_1,\ldots,g_S)$ is the collection of the generators
for the different sources.
The goal of the weak supervision and strong supervision terms
$\mathcal{F}^W_i$ and $\mathcal{F}^S_i$ will be to obtain the
best possible fits for clean data from the different sources. 
In contrast, the goal of the adversarial term $\mathcal{F}^A_i$
is to fit certain adversarial data as bad as possible. We call $\tau_W$ the weak supervision weight,
$\tau_A$ the adversarial weight, and $\tau_S$ the strong supervision weight. 
In the case where no strong supervision data is available,
we can ignore the term $\mathcal{F}^S_i$ in~\eqref{eq:full_model};
formally this can be done by setting the regularization parameter $\tau_S$ to $0$.

We now discuss the definitions of the different terms.

\begin{itemize}
\item \emph{Weak supervision term:} 
  Here the goal is to fit data from source $i$
  as good as possible with the model $g_i$. Thus it makes sense
  to minimize the reconstruction error $\lVert \pi_i(g_i,u) - u\rVert^2$
  over all available samples $u$ from the source $i$,
  where $\pi_i$ is defined in~\eqref{eq:projection}.
  Taking into account all the available training data from source $i$,
  we thus arrive at the definition
  \begin{equation}\label{eq:FW}
    \mathcal{F}^W_i(g_i) := \mathbb{E}_{u \sim \mathbb{P}_{U_i}} \bigl(\lVert \pi_i(g_i,u) - u \rVert^2\bigr),
  \end{equation}
  where $\mathbb{E}_{u\sim\mathbb{P}_{U_i}}(\cdot)$ denotes the expectation
  with respect to the probability distribution $\mathbb{P}_{U_i}$.
  This is the only term included in standard training of generative models, where the goal is only to 
  represent data as well as possible, not necessarily to be used for downstream tasks like source separation.
\item \emph{Strong supervision term:}
  Here we are working with the available samples
  of correctly separated data $v = \sum_i a_i u_i$.
  As a cost function, it makes sense to use the mean reconstruction error
  \begin{equation}\label{eq:FS}
    \mathcal{F}^S_i(\mathbf{g}) :=
    \mathbb{E}_{(\mathbf{a},\mathbf{u};v) \sim \mathbb{P}_{\mathcal{A}\times \mathcal{U}\times V}} \bigl(\lVert u_i^*(\mathbf{g};v) - u_i \rVert^2\bigr),
  \end{equation}
  where $u_i^*$ is as defined in~\eqref{eq:decomposition}.
  In case the solution of~\eqref{eq:decomposition} is not unique,
  we use in the definition of $\mathcal{F}^S_i$ a solution
  for which $\lVert u_i^*(\mathbf{g};v)-u_i\rVert$ is minimal.  
  Note that this cost function depends not only on the generator $g_i$,
  but also the generators $g_j$ for $j \neq i$.
  If this is the only term is used, we can interpret this as a 
  particular parameterization of a discriminative end-to-end mapping, 
  especially if the encoder in equation~\eqref{eq:decoder} is trained alongside $\mathbf{g}$.
\item \emph{Adversarial term:}
  Here the goal is to fit certain \emph{adversarial data}
  as badly as possible with the model $g_i$.
  Noting that this term appears with negative sign in
  the functional $\mathcal{F}_i$, we can thus define
  \[
    \mathcal{F}^A_i(g_i) := \mathbb{E}_{u \sim \mathbb{P}_{Z_i}} \bigl(\lVert \pi_i(g_i,u) - u \rVert^2\bigr),
  \]
  where $\mathbb{P}_{Z_i}$ is the probability distribution
  of the adversarial data.

  For the definition of the adversarial data,
  we follow the argumentation of~\cite{lunz2018adversarial}
  and include data that is produced by naively inverting
  the forward operator $A$, as this approach is expected
  to preserve, or even amplify, possible noise in the data.
  For this we define
  \begin{equation}\label{eq:fi}
    f_i(a_1,\ldots,a_S;v) = \frac{a_i}{\sum_{j = 1}^S a_j^2} v,
  \end{equation}
  which is precisely the $i$-th component of $A^\dagger v$,
  where $A^\dagger$ is the pseudo-inverse of $A$.
  Denote moreover by $\mathbb{P}_{\mathcal{A}\times V}$ the
  joint distribution of weights and mixed data.
  Then we obtain the distribution $\mathbb{P}_{V_i}$ of
  the $i$-th component of naively inverted mixed data
  as
  \[
    \mathbb{P}_{V_i} := (f_i)_\# (\mathbb{P}_{\mathcal{A}\times V}),
  \]
  where $(f_i)_{\#}$ denotes the push-forward by the mapping $f_i$.
  That is,
  $\mathbb{P}_{V_i}(\Omega) = \mathbb{P}_{\mathcal{A}\times V}(f_i^{-1}(\Omega))$
  for every measurable subset $\Omega \subset \R^m$ of mixed data.
  In the weak supervision setting, we cannot generally assume
  that the joint distribution $\mathbb{P}_{\mathcal{A}\times V}$ is available.
  Lacking any better model, we will therefore assume that
  $\mathbb{P}_{\mathcal{A}}$ and $\mathbb{P}_V$ are independent, which naturally
  encapsulates the case where the weights are deterministic.
  \smallskip
  
  In addition to naively inverted mixed data, we propose
  to include data from other sources $j \neq i$
  in the adversarial data for the training of $g_i$,
  as we intuitively want the generator $g_i$ be bad at
  representing different sources.
  Thus, we define the distribution of the adversarial data as the mixture distribution
  \[
    \mathbb{P}_{Z_i} := \sum_{j\neq i} \omega_{ij} \mathbb{P}_{U_j} + \omega_{ii} \mathbb{P}_{V_i},
  \]
  where the weights $0 \le \omega_{ij} \le 1$ are chosen such that
  $\sum_{j} \omega_{ij} = 1$.
  In practice, a reasonable possibility is to weigh the different
  terms according to the amount available data.
  This leads to the adversarial distribution  
  \[
    \mathbb{P}_{Z_i} := \sum_{j\neq i} \frac{N_j}{\hat{N}_i} \mathbb{P}_{{U}_j}
    + \Bigl(1-\sum_{j\neq i} \frac{N_j}{\hat{N}_i}\Bigr)\mathbb{P}_{V_i},
  \]
  where $N_j$, $j=1,\ldots,S$, is the amount of weak supervision data from the $j$-th source,
  $N_V$ the amount of available mixed data,
  and $\hat{N}_i := N_V + \sum_{ \neq i} N_j$ the total amount of adversarial data for the $i$-th source.
\end{itemize}

Combining these terms, we arrive at the cost function
\begin{multline}\label{eq:fullF}
  \mathcal{F}_i(\mathbf{g}) =
  \tau_W\mathbb{E}_{u \sim \mathbb{P}_{U_i}} \bigl(\lVert \pi_i(g_i,u) - u \rVert^2\bigr)
  - \tau_A \mathbb{E}_{u\sim \mathbb{P}_{Z_i}} \bigl(\lVert \pi_i(g_i,u) - u \rVert^2\bigr) \\
  + \tau_S \mathbb{E}_{(\mathbf{a},\mathbf{u};v) \sim \mathbb{P}_{\mathcal{A} \times \mathcal{U}\times V}} \bigl(\lVert u_i^*(\mathbf{g};v) - u_i \rVert^2\bigr)
\end{multline}
for the training of the $i$-th generator function $g_i$.
We would like to stress that the inclusion of an adversarial term for training generative models for use in regularization
is a novel concept to our know\-ledge. We call this approach \textbf{Maximum Discrepancy Generative Regularization}.

In the case where no strong supervision data is available,
we set $\tau_S := 0$, effectively ignoring the last term.
In that case, the function $\mathcal{F}_i$ only depends on the
generator $g_i$, and thus the training of the different generators
can be performed independently.

In the presence of strong supervision data, the different
cost functions are coupled through the terms $u_i^*(g_1,\ldots,g_S;v)$
that occur in the strong supervision costs $\mathcal{F}^S_i$.
In that case, we obtain therefore the multi-objective problem
of solving
\[
  \min_{\mathbf{g}}\bigl\{\mathcal{F}_1(\mathbf{g}),\ldots,\mathcal{F}_S(\mathbf{g})\bigr\}.
\]
One approach for this is to minimize the weighted sum
\begin{equation}\label{eq:weightedsum}
  \min_{\mathbf{g}} \Bigl(\alpha_1 \mathcal{F}_1(\mathbf{g}) + \ldots + \alpha_S \mathcal{F}_S(\mathbf{g})\Bigr)
\end{equation}
with weights $\alpha_i > 0$ depending on the importance of the source $i$
as well as the available training data for that source.

We will now show that the problem~\eqref{eq:weightedsum} admits a solution.
For this, we will need a technical assumption concerning the projections $\pi_i$.

\begin{assumption}\label{ass:2}
  For all $g_i \in \mathcal{G}_i$ and $u \in \R^m$,
  the value $\lVert \pi_i(g_i,u)-u\rVert$ is independent of the choice
  of $h_i^*(g_i,u)$.
\end{assumption}

We note that Assumption~\ref{ass:2} is trivially satisfied if
$\pi_i(g_i,u)$ is unique. This is in particular the case
for Non-negative Matrix Factorization, to be discussed in Section~\ref{sec:NMF}.

\begin{theorem}
  \label{theorem:existence}
  Assume that Assumptions~\ref{ass:1} and~\ref{ass:2} hold.
  In addition, assume that the sets $\mathcal{G}_i$ are (sequentially) compact and
  that the training data have finite second moments,
  that is,
  \[
    \mathbb{E}_{u\sim\mathbb{P}_{U_i}}(\lVert u \rVert^2) < \infty
  \]
  for all $i=1,\ldots,S$ and that
  \[
    \mathbb{E}_{v \sim \mathbb{P}_V}(\lVert v \rVert^2) < \infty.
  \]
   
  Then the problem~\eqref{eq:weightedsum} admits a solution.
\end{theorem}

The proof is given in Appendix~\ref{app:proofs_1}.

\begin{remark}\label{re:compactness}
  For dictionary methods like NMF, compactness of the sets $\mathcal{G}_i$
  required in Theorem~\ref{theorem:existence}
  can be obtained by normalizing the bases;
  for neural networks, one can constrain the weights of the network.
  For GANs, this is commonly done to the discriminator with weight clipping, spectral normalization
  and gradient penalties \cite{li2023systematic}. It is less commonly
  also applied to the generator itself to stabilize training.
  As an alternative, it is possible to choose a (coercive, proper, and lower semi-continuous)
  regularization term $\mathcal{T}$ for the generators and train them by minimizing the
  regularized functional
  \[
    \min_{\mathbf{g}} \Bigl(\alpha_1 \mathcal{F}_1(\mathbf{g}) + \ldots + \alpha_S \mathcal{F}_S(\mathbf{g})
    + \mathcal{T}(\mathbf{g})
    \Bigr).
  \]
  With straightforward modifications of the proof of Theorem~\ref{theorem:existence}
  the existence of solutions of this problem can be shown as well.
\end{remark}

\section{Application to NMF}\label{sec:NMF}

Non-negative Matrix Factorization (NMF) is a matrix factorization and dictionary learning method 
traditionally used for dimensionality reduction \cite{lee1999learning}, and with notable
applications in source separation problems \cite{fevotte2018single}.
Assume that we are given training data in the form of samples
$u_i^{(k)} \in \mathbb{R}_{\ge0}^m$ for each source $i$.
The underlying assumption for NMF is that these samples can be approximately
written as non-negative linear combinations
$u_i^{(k)} \approx W_i h_i^{(k)}$ for some yet to be determined basis/dictionary
$W_i \in \mathbb{R}_{\ge 0}^{m\times d}$ and latent variables/weights $h_i^{(k)} \in \mathbb{R}_{\ge0}^d$.
Here the number of basis vectors $d$ has to be chosen a-priori and will have
a large effect on the final result, as well as the computational complexity.
It is also possible to choose a different dimension $d_i$ for each of the sources,
which can be important if, for instance, one of the source signals is vastly more
complex than the others.

One of the main drawbacks of NMF is that the underlying assumption is very restrictive
and NMF may therefore not be sufficiently expressive for real data.
One way of solving this problem is to train NMF with more basis vectors.
However, this leads to overdetermined bases that can represent signals of other sources, which in turn leads to poor separation results.
In the literature, the main way of alleviating this is using sparsity during testing and training, 
so that a signal can be represented with a small number of basis vectors
in a large dictionary \cite{le2015sparse}. 
Instead of fitting sparse NMF, we can also fit so-called Exemplar-based NMF (ENMF), where the basis vectors are simply chosen as randomly sampled signals from the training set \cite{weninger2014discriminative}.
ENMF has the benefit that it requires no training (only sampling), while still providing comparable performance for large $d_i$. 
Usually, one additionally assumes sparsity in the weights, that is, many of the components of $h_i^{(k)}$.

Since NMF is a traditional source separation method, and because of its relative low complexity,
we will use it as an application to explore the ideas of maximum discrepancy generative regularization.

\subsection{NMF as generative model for source separation}
\label{ss:NMFSS}

We will now explain how NMF can be used for single channel source separation
and how it fits into the approach described in Section~\ref{se:intro},
see also \cite{le2015sparse,weninger2014discriminative}.
We start by interpreting NMF as a generative model,
defining the generator $g_i$ as
\[
  g_i(h) := W_i h
  \qquad\text{ where }
  W_i \in \R_{\ge 0}^{m \times d_i}
  \text{ and }
  h \in \R^{d_i}_{\ge 0}.
\]
In the following, we will identify the generator $g_i$ with
the non-negative matrix $W_i$.

On a more abstract level, we can
interpret the NMF model as the assumption that the
data samples lie in the positive (sparse) cone spanned by
the columns of $W_i$.
This cone is invariant under scaling of the columns of $W_i$,
and thus we may assume without loss of generality that
they are all normalized to length $1$.

\subsubsection*{Separation of new data}

A common regularization term used in the context of NMF
is the entry-wise $1$-norm
\[
  \mathcal{R}(h) := \lvert h \rvert_1,
\]
which encourages sparsity in the representations of the data.
Assuming that we have trained generators $W_i$ for the different sources,
we can thus separate a mixed signal $v \in \mathbb{R}_{\ge 0}^m$ by solving the problem
(cf.~\eqref{eq:TikGen})
\begin{equation}
  (h_i^\ast)_{i=1}^S = \argmin_{h_i \ge 0,\, i = 1,\ldots,S}
  \Bigl\lVert v - \sum_{i = 1}^S W_i h_i\Bigr\rVert_F^2 + \lambda \sum_i \lvert h_i \rvert_1
  \label{eq:gentest}
\end{equation}
and recover the separated signals as $u_i^\ast = W_i h_i^\ast$.
Since the NMF generates signals in a convex cone,
which is invariant under scaling,
we have here ignored the possible scaling of the different sources.

For the numerical solution of~\eqref{eq:gentest}, we can concatenate the bases and the latent
variables to matrices
\begin{equation*}
  W = \begin{bmatrix}W_1 & \hdots & W_S \end{bmatrix},
  \quad
  h = \begin{bmatrix}h_1^T & \hdots & h_S^T \end{bmatrix}^T.
\end{equation*}
Then~\eqref{eq:gentest} reduces to a
non-negative least squares problem with sparsity constraints equivalent to~\eqref{eq:NMFtrainingH}.
  
Finally, it is common to apply afterwards a Wiener filter
\begin{equation}
  \label{eq:wiener}
  u_i = v \odot \frac{W_i h_i^\ast}{\sum_{j = 1}^S W_jh_j^\ast + \varepsilon},
\end{equation}
where $\odot$ denotes entry\-wise (Hadamard) product the division is interpreted entry\-wise and $\varepsilon$ is a parameter representing the magnitude of the residual $e = v - \sum_{i = 1}^S u_i$, which is assumed to follow a Gaussian distribution.
We will throughout this paper assume that this residual is near zero, and instead use $\varepsilon$ as a safe division factor.
This post-processing is crucial for NMF-based methods to perform well as it imposes data fidelty, and the constraint that
data lies on a convex cone is too restrictive for practical applications.

\subsubsection*{Weakly supervised training}

We will now discuss how NMF can be trained in the presence
of weak supervision data in the terminology of Section~\ref{ss:adversarial_training}.
Historically, this approach has been named ``supervised NMF'' \cite{fevotte2018single}, to differentiate it from using NMF for blind source separation.
We instead denote this in the following as \emph{Standard NMF} to distinguish it from NMF models that utilize strong supervision data.

To start with, we collect the samples $u_i^{(k)} \in \mathbb{R}^m$
for each source $i$ column-wise in a matrix $U_i$.
For each $i$, we then solve the bi-level problem
\begin{equation}
  \min_{W_i \ge 0} \|U_i - W_iH(W_i,U_i)\|_F^2,\label{eq:NMFtrainingW}
\end{equation}
where
\begin{equation}
  H(W_i,U_i) = \argmin_{H \ge 0} \|U_i - W_iH\|_F^2 + \lambda |H|_1.\label{eq:NMFtrainingH}
\end{equation}
Here $\lVert \cdot \rVert_F$ denotes the Frobenius norm,
$\lvert \cdot \rvert_1$ denotes the entry-wise $1$-norm,
and $\lambda > 0$ is a parameter controlling the
sparsity of the matrices $H_i$.
It is also possible to add a further regularization term
for the generators $W_i$ in order to, for instance,
enforce representations with sparse vectors only.
Also, it is possible to choose different regularization parameters
for the different sources.
With the notation of Section~\ref{ss:adversarial_training},
the training problem~\eqref{eq:NMFtrainingW} is the same as
minimizing the weak supervision terms $\mathcal{F}_i^W$ defined in~\eqref{eq:FW}.
We note that this bi-level formulation is not standard in NMF literature.

\medskip

In the case where we only have samples from the sources $1,\ldots,S-1$,
but not from source $S$, the training is slightly different:
For the training of the bases $W_i$, $i=1,\ldots,S-1$, of the given
sources we proceed as above.
Then, however, we estimate a basis $W_S$ for the last source
by trying to fit the given samples to mixed data.
To that end, we store the mixed data column\-wise in a matrix $V$
and compute $W_S$ by solving
\begin{equation}
  \min_{\substack{W_S \ge 0,\\H_i\ge 0,\, i=1,\ldots,S}}
  \Bigl\lVert V - W_S H_S - \sum_{i=1}^{S-1} W_i H_i\Bigr\lVert_F^2
  + \lambda\sum_{i=1}^S \lvert H_i\rvert_1.
  \label{eq:semi}
\end{equation}
How well $W_S$ is able to approximate the unknown source $S$ then depends on the
quality of the other bases $W_i$, $i=1,\ldots,S-1$, as well as the amount
of available mixed data.

We can also extend this approach to the case where strong supervision data is available
by including a term as defined in~\eqref{eq:FS}.
Such a generalization has been proposed by other authors \cite{weninger2014discriminative} under the name \emph{Discriminative NMF} (DNMF).
To be more precise, in their work, they further propose optimizing the Wiener-filtered solution given in equation~\eqref{eq:wiener}, as well as using different NMF
dictionaries for separation and for reconstruction, that is, separate dictionaries for the decoder defined in~\eqref{eq:decoder} and for the encoder.
We will, however, focus instead on the approach outlined in previous chapters, denoting the case when we only
use the strong supervision term as DNMF.

\subsection{Maximum Discrepancy NMF}
\label{sec:ANMF}

We now combine the standard approaches to the training of NMF bases
with the idea of adversarial or maximum discrepancy training
developed in Section~\ref{ss:adversarial_training}.
To start with, we consider the setting where no strong supervision data is available.
Here we obtain the problem
\begin{multline*}
  \min_{W \ge 0} \bigl(\mathcal{F}^W_i(W_i) - \mathcal{F}^A_i(W_i)\bigr) \\
  = \tau_W \mathbb{E}_{u \sim \mathbb{P}_{U_i}}[\lVert u - Wh(W,u)\rVert^2] -  \tau_A \mathbb{E}_{u \sim \mathbb{P}_{Z_i}}[\lVert u - Wh(W,u)\rVert^2] \\
  \approx \min_{W \ge 0} \frac{\tau_W}{N_i} \lVert U_i - W_iH(W_i, U_i)\rVert_F^2 - \frac{\tau_A}{\hat{N}_i} \lVert \hat{U}_i - W_iH(W_i,\hat{U}_i)\rVert_F^2.
\end{multline*}
Here the second line is obtained by Monte-Carlo integration, and $U_i$ and $\hat{U}_i$ are data stored columnwise
and sampled from $\mathbb{P}_{U_i}$ and $\mathbb{P}_{Z_i}$, respectively.
In our framework, the adversarial data is the concatenation of all the data from other sources
$U_j$, $j \neq i$, and naively inverted mixed data.
We refer to Appendix~\ref{app:scaling} for more details on the generation and sampling
of adversarial data.

We note that, for equal weights $\tau_W = \tau_A = 1$,
the quantity $-\min_{W_i \ge 0} \mathcal{F}^W_i(W_i) - \mathcal{F}^A_i(W_i)$, viewed as function of $\mathbb{P}_{U_i}$
and $\mathbb{P}_{Z_i}$,
is a form of \emph{Integral Probability Metric} (IPM),
which is a notion that generalizes the Wasserstein distance, see~\cite{muller1997integral}.
In view of the similarity to another IPM, namely the Maximum Mean Discrepancy \cite{gretton2012kernel},
we call this framework \textbf{Maximum Discrepancy NMF} (MDNMF).

An illustration of the conceptual difference between NMF and MDNMF is shown in Figure~\ref{fig:NMF_example}.

\begin{figure}
  \includegraphics[width=0.49\textwidth]{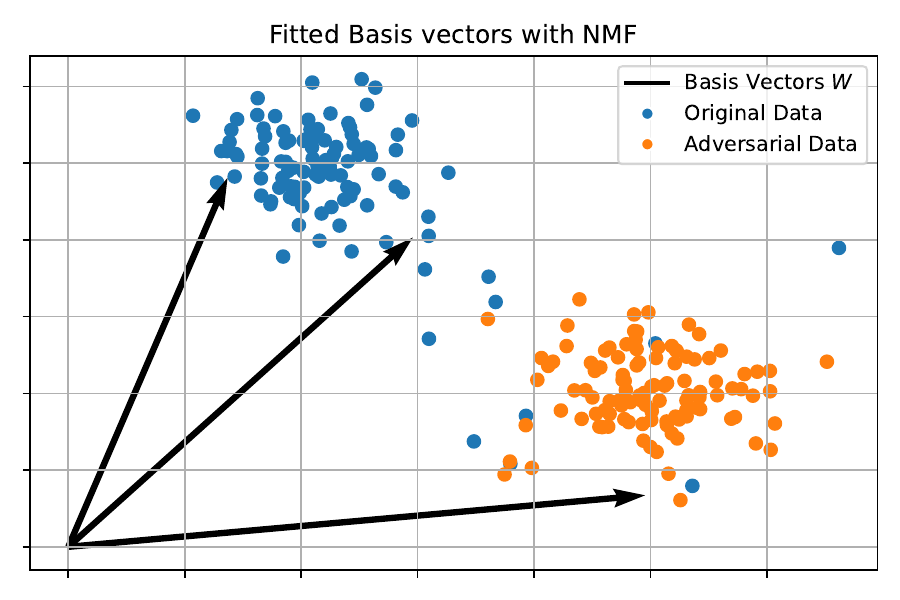}
  \includegraphics[width=0.49\textwidth]{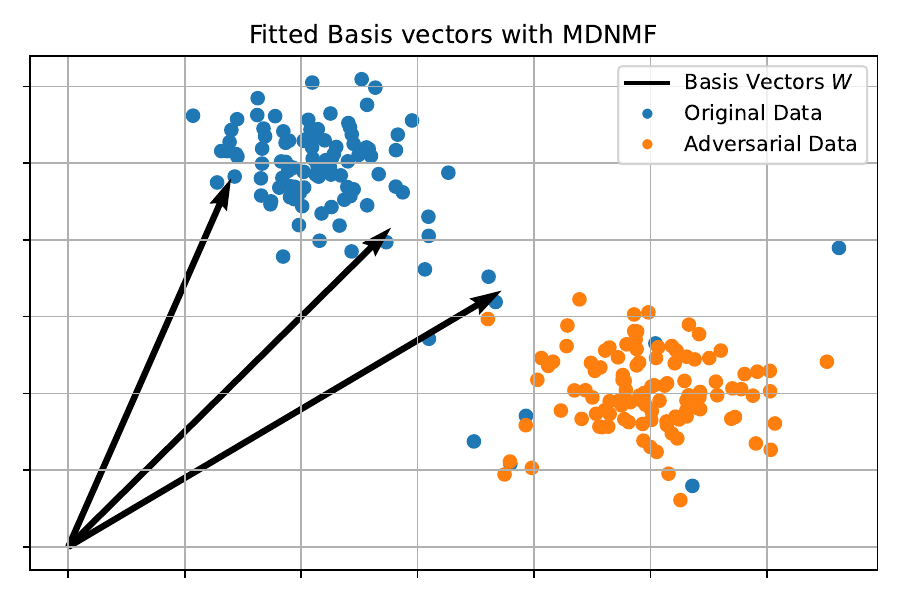}
  \caption{A figure illustrating the difference between NMF and MDNMF. Here we use $\lambda = [2\cdot10^{-2}, 2\cdot10^{-2}, 3\cdot10^{-2}]$, $\tau_W = 1$ and $\tau_A = \sqrt{0.25}$. 
  A fitted basis with NMF can be sensitive to outliers, and the resulting
  basis can also represent adversarial data. MDNMF explicitly avoids fitting adversarial data, potentially leading to worse representation
  of outlier data. This property can be beneficial for downstream tasks like source separation.  }
  \label{fig:NMF_example}
\end{figure}

While MDNMF can be fitted with sparsity, it can also be used to select a
small number of basis vectors that yield good separation
as opposed to selecting an overdetermined sparse basis.

\subsection{Discriminative and Maximum Discrepancy NMF}
\label{sec:FNMF}

Combining MDNMF with DNMF we obtain the full model defined in equation~\eqref{eq:fullF}.
We call this approach \emph{Discriminative and Maximal Discrepancy NMF} (D+MDNMF), and it encompasses the NMF models discussed so far,
which can be recovered by setting certain parameters to $0$ or $1$,
see Table \ref{tab:NMF}.
\begin{table}[!hbt]
\centering
\begin{tabular}{c|cccc}
& NMF & MDNMF & DNMF & D+MDNMF \\ 
\hline
$\tau_W$ & $=1$ & $=1$ & $=0$ & $>0$ \\
$\tau_A$ & $=0$ & $>0$ & $=0$ & $>0$ \\ 
$\tau_S$ & $=0$ & $=0$ & $=1$ & $>0$
\end{tabular}
\caption{An illustration of the parameter values for $\tau_W$, $\tau_A$, and $\tau_S$ for the different variations of NMF. 
D+MDNMF is a superset of all methods. }
\label{tab:NMF}
\vspace{-\baselineskip} 
\end{table}

This approach is most interesting in the case where only a small amount
of strong supervised data is available as compared to weak supervised data, as fitting purely adversarially or discriminatively is most likely better when a large amount of strong supervised data is available.
We also have the option of using the strong supervised data for fitting
both the weak supervision terms and strong supervision terms. For D+MDNMF, this would mean that $\tau_S$ represents the relative weight of weak supervision compared to
strong supervision fitting. This can potentially alleviate the problem of overfitting to specific data.

\section{Numerical implementation}
\label{sec:numerics}

\subsection{Multiplicative updates}

One standard approach to fitting NMF is to alternate between updates of the basis while keeping the latent variables fixed, 
and updates of the latent variables while keeping the basis fixed. This approach can usually find a local minimizer \cite{lee2000algorithms}.

A standard multiplicative update for finding $H(U,W)$ in equation~\eqref{eq:NMFtrainingH} where $U \in \mathbb{R}_+^{m \times N}$ is given by
\begin{equation}
H \leftarrow H \odot \frac{W^TU}{W^TWH + \lambda}.
\label{eq:Hupdate}
\end{equation}
Here $\lambda > 0$ is the regularization parameter, but also serves as a safe division factor \cite{lee2000algorithms}.
This update preserves non-negativity, and the loss we are optimizing for
is non-increasing under this update.
We can obtain a similar update rule for $W$, which will be introduced later, and solve the NMF training problem \eqref{eq:NMFtrainingW}
by applying alternate updates to $H$ and $W$.
Although a good alternative to \eqref{eq:Hupdate} is to use projected gradient methods, we choose to use multiplicative updates because of their simplicity.

\subsubsection*{Multiplicative update for D+MDNMF}

We denote $H_i$, $\hat{H}_i$ and $\tilde{H}_i$ as the latent variable of the weak supervision, adversarial and strong supervision data of the $i$-th source respectively.
We use corresponding notation for the data.
The full update for D+MDNMF for the $i$-th source is 
\begin{equation}
  W_i \leftarrow W_i \odot \frac{\tau_W U_iH_i^T/N_i + \tau_A W_i \hat{H}_i\hat{H}_i^T/\hat{N}_i + \tau_S \tilde{U}_i \tilde{H}_i^T/\tilde{N}}{\tau_W W_i H_iH_i^T/N_i + \tau_A\hat{U}_i\hat{H}_i^T/\hat{N}_i + \tau_S W_i \tilde{H}_i \tilde{H}_i^T/\tilde{N}_i + \gamma}.
  \label{eq:DANMFW}
\end{equation}
Here the parameter $\gamma > 0$ is a safe division factor ensuring that the denominator never vanishes;
alternatively, this can be interpreted as a sparsity regularization parameter for $W_i$.
The corresponding updates for the latent variables are given by applying equation \eqref{eq:Hupdate} to the appropriate data:
$H_i = H(U_i,W_i)$, $\hat{H}_i = H(\hat{U}_i,W_i)$ and $\tilde{H} = H(V, W)$. 
See Appendix~\ref{app:scaling} for extra notes on how to scale the adversarial data.
In addition, we normalize the columns of $W_i$ after each update step,
followed by a corresponding rescaling of the weights $H_i$.

Similar to the standard $W$ update for NMF \cite{lee2000algorithms}, the update~\eqref{eq:DANMFW} obviously preserves non-negativity, 
and we can further show that the loss we are optimizing is non-increasing with the update.
\begin{theorem}
  \label{theorem:nonincrease}
  The D+MDNMF loss defined in equation~\eqref{eq:fullF} is non-increasing under the update~\eqref{eq:DANMFW}.
\end{theorem}
The proof is given in Appendix~\ref{app:proofs_2}.
\begin{remark}
 Theorem~\ref{theorem:nonincrease} does not imply that the update guaranteed to converge to a local minimizer of~\eqref{eq:fullF}. 
 We show numerical convergence experimentally in Appendix~\ref{app:convergence}.
\end{remark}

The asymptotic computational complexity of each iteration of D+ANMF for all bases
is of order $\mathcal{O}(dmN_\text{tot})$, where $N_\text{tot} = \sum_{i = 1}^S [N_i +\hat{N}_i + N_\text{sup}]$ is the total amount of data.
Thus the computational complexity of the updates for NMF, MDNMF, DNMF, and D+MDNMF scales at the same rate with the amount of data used.  

\subsubsection*{Semi-supervised update}

In the semi-supervised case where the $S$-th source is unknown, but we have access to the mixed data $V$, the updates for $W_S$ and the latent variables $H_i$ become
\begin{align*}
  W_S &\leftarrow W_S \odot \frac{V H_S^T/N_V}{(\sum_{i = 1}^S W_i H_i) H_S^T/N_V + \gamma}, \\
  H_i &\leftarrow H_i \odot \frac{W_i^T V}{W_i^T(\sum_{j = 1}^S W_j H_j) + \lambda}.
\end{align*}
Here the bases $W_i$, $i = 1,\ldots,S-1$, are pre-trained, and they can be trained adversarially by using the mixed data as adversarial data.

\subsection{Stochastic Multiplicative Updates}
\label{sec:multiplicative_update}

For standard NMF, each column of the data $U$ has a corresponding column in the latent variable $H$. At the start of each epoch,
we can shuffle the data in $U$, perform a corresponding shuffle of $H$, and divide the matrices column\-wise into batches $U^{(b)}$ and $H^{(b)}$ following the ideas of \cite{serizel2016mini}.
We can then successively apply the update for $W$ for data from the different batches.
We update all latent variables $H$ simultaneously instead of batch-wise in a single update.

For initialization, we primarily use exemplar-based initialization, which means that to initialize $W$ so that $U \approx WH$, we sample $d$ columns of $U$, and then initialize $H$ via projection.
This is also used in the cases where $W$ is already pre-trained, and we want to train either adversarially or discriminatively. In the cases where this is infeasible, we use randomized initialization, where we sample from a uniform distribution. 

When applying these stochastic multiplicative updates to either MDNMF or D+MDNMF, we face the challenge that we are minimizing a loss with
different terms and potentially unbalanced data, and need to select batch sizes accordingly. 
To overcome this, we select one term we are interested in fully sampling and undersample or oversample data from the other terms.
The data is shuffled and a new epoch begins when we have passed through all data of this chosen term.
For more discussion on this, see Appendix~\ref{app:numerical}.
The full proposed algorithm for fitting D+MDNMF where the same parameters are used for all sources is given in Algorithm~\ref{alg:smu}. For simplicity, we assume that
all parameters are the same for the different sources, including the number of basis vectors and the amount of data in each source.

\setcounter{algorithm}{0}
\begin{algorithm}
\caption{Stochastic Multiplicative Update for D+ANMF}
\label{alg:smu}
\begin{algorithmic}
  \STATE \textbf{Input: } $\text{epochs} \in \mathbb{N}$, $d \in \mathbb{N}$, $\lambda$, $\gamma > 0$, $\tau_W$, $\tau_A$, $\tau_S > 0$, and batch sizes.
  \STATE \textbf{Data input:} True, adversarial, and supervised datasets $U \in \mathbb{R}_+^{S \times m \times N}$, $\hat{U} \in \mathbb{R}_+^{S \times m \times \hat{N}}$, and $\tilde{U} \in \mathbb{R}_+^{S \times m \times N_{\text{sup}}}$. Supervised mixed data $\tilde{V} \in \mathbb{R}_+^{m \times N_\text{sup}}$.
  \STATE \textbf{Initialize: } $W \in \mathbb{R}_+^{m \times dS}$ randomly or exemplar-based.
  \STATE \textbf{Initialize: } Latent variables $H$, $\hat{H}$, $\tilde{H}$ either randomly or with $\eqref{eq:Hupdate}$ applied to the respective data.
  \STATE \textbf{Calculate: } Number of batches.
  \FOR{$k = 0, k < \text{epochs}$}
    \STATE Shuffle $U, \hat{U}, \tilde{U}, H, \tilde{H}, \hat{H}, \tilde{V}$.
    \STATE Update $\tilde{H}$ with \eqref{eq:Hupdate} applied to $H(\tilde{V}, W)$.
    \FOR{$i = 1, i \le S$}
      \STATE Update $H_i$ and $\hat{H}_i$ with \eqref{eq:Hupdate} applied to $H(U_i, W_i)$ and $H(\hat{U}_i,W_i)$ 
      \FOR{$b = 0, b < \text{number of batches}$}
        \STATE Update $W_i$ with \eqref{eq:DANMFW} using $U_i^{(b)}$, $H_i^{(b)}$, $\hat{U}_i^{(b)}$, $\hat{H}_i^{(b)}$, $\tilde{U}_i^{(b)}$ and $\tilde{H}_i^{(b)}$.
      \ENDFOR
    \ENDFOR
    \STATE Normalize $W$ and rescale $H$, $\hat{H}$, $\tilde{H}$.
  \ENDFOR
\end{algorithmic}
\end{algorithm}

This universal algorithm for D+MDNMF can also be used to fit NMF, MDNMF and D+MDNMF by selecting the parameters $\tau_W$, $\tau_A$, and $\tau_S$, see Table~\ref{tab:NMF}. 
It is also worth noting that we can swap the order of the loop over sources and the loop over epochs. For NMF and MDNMF this does not affect anything,
as the bases can be fitted independently of each other in parallel. This should however not be done for DNMF and D+MDNMF, as the bases should be updated concurrently for each epoch.

\subsection{Hyperparameter tuning}
\label{sec:hyperparameter}

Hyperparameter refers to a parameter that is chosen a priori and is not fitted during training. This usually needs
to be done in the presence of strong supervised data, even if the models themselves do not need to be trained with strong supervised data.
In this sense, performing hyperparameter tuning fits a discriminative model, because the model is tuned to a specific problem.

We will be interested in fitting several parameters at the same time, and we will do this where limited amounts of strong supervised data is available.
For this, we will apply a randomized search, as it is easy to implement while still being relatively efficient \cite{bergstra2012random}.
The idea of this method is to simply sample hyperparameters from predetermined probability distributions and select the parameters
that yield the best solution on the test data with or without cross-validation (CV) \cite{hastie01statisticallearning}.

For NMF and MDNMF, which only require weak supervised data, we have the option of using the strong supervised
data for fitting or only for tuning. We expect that the benefits of using more data outweighs the detriments of overfitting,
and therefore use all data for fitting. 

For DNMF and D+MDNMF, we use CV to avoid overfitting the specific strong supervised data.
For D+MDNMF, each CV fold uses all weak supervision data available for fitting.

For source separation problems, the selection of a suitable metric
to be optimized during hyperparameter tuning
is a non-trivial problem, because we have to assess the
quality of all the separated sources simultaneously.
For that, we suggest using a weighted mean of some metric of interest
(PSNR, SDR, etc.) over the sources, where the weight is chosen
based on the importance of the sources.
When the signals are of equal importance, we
will use the arithmetic mean,
and for denoising problems we will ignore the noise source
and only weigh the signal of interest. 

\section{Numerical experiments}
\label{sec:numerical}

We now want to test our proposed algorithm for both image and audio data with a few different data settings.
For all experiments we implement the algorithms in Python using NumPy \cite{harris2020array};
the code
is available in the GitHub repository \url{https://github.com/martilud/ANMF}.

\subsection{Image data}
We first test our algorithms on the famous MNIST dataset \cite{deng2012mnist}. 
This dataset consists of $70000$ $28\times28$ grayscale images of $10$ different handwritten digits. 
We treat each of the different digits as a class and attempt to separate mixed images
of two different digits from each other.

\subsubsection*{Experiment 1: Data Rich, Strong Supervised Setting}

We investigate first a setting with an abundance of strong supervision data,
which make discriminative models like DNMF applicable.
We use $N_{\text{sup}} = 5000$ data-points and synthetically generate strong supervised data with ``zero'' and ``one'' digits with deterministic weights $a_0 = a_1 = 0.5$.
Similarly, we create $N_{\text{test}} = 1000$ data points that will be used for testing. 
We select the same number $d$ of basis vectors for both sources, and test different values.
We select the sparsity parameter as $\lambda = 1\mathrm{e}-2$ and choose
the safe division factor $\gamma = 1\mathrm{e}-10$, as this choice worked well for all methods.

The results are illustrated in Figure~\ref{fig:data_rich}. We see that MDNMF outperforms all methods, and
performance increases with the number of basis vectors $d$. 
\begin{figure}[!htb]
\centering
\includegraphics[width = 0.8\textwidth]{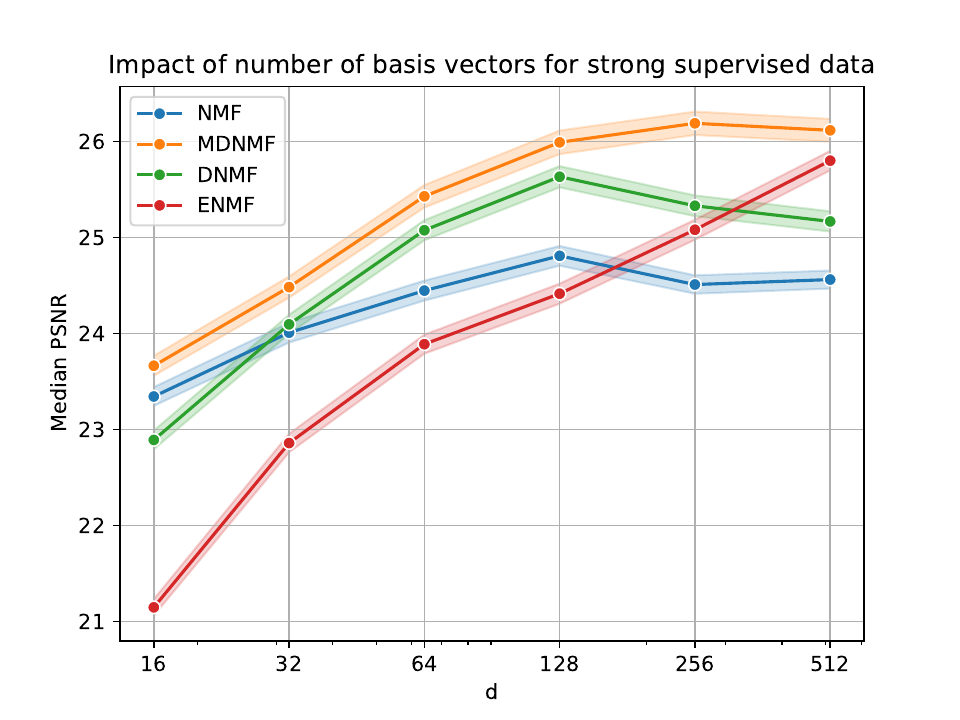}
\caption{Results from experiments in a data rich strong supervised setting. The lines show the median PSNR of the reconstructions of the methods applied to the test dataset, along with the standard error.
For MDNMF we use $\tau_W = 1$ and $\tau_A = 0.2$.
We note that performance tends to improve as the number $d$ of basis vectors increases, and that MDNMF consistently outperforms the other methods.}
\label{fig:data_rich}
\vspace{-\baselineskip} 
\end{figure}

Surprisingly, MDNMF outperforms DNMF in this strong supervision setting,
even though it does not explicitly utilize the fact that we have strong supervision data, which DNMF explicitly does.
We predict that, given enough data, DNMF can potentially find a global minimizer that outperforms MDNMF, though this does not practically happen in our experiments.

One beneficial property of MDNMF is that it performs model selection in the sense that as $d$ increases, additional basis vectors are chosen to be in the span of the existing ones,
and these can potentially be removed after fitting. This is seen in Figure~\ref{fig:data_rich}, as performance plateaus for high $d$ for MDNMF. It thus selects a ``sharp'' cone and can potentially be used to select optimal $d$. 

We observe that ENMF can outperform NMF for large $d$. This is remarkable, as ENMF requires
virtually no training, only sampling.
However, we are mainly interested in the case of lower values of $d$,
where ENMF performs poorly, we will omit the results of ENMF in the further experiments.

We now investigate how the performance is affected by the value of the parameter $\tau_A$. We run the same experiment,
except this time we vary $\tau_A$ and focus only on MDNMF. The results are shown in Figure~\ref{fig:data_rich_tau}.
They indicate that, while a good selection of the parameter $\tau_A$ is crucial for performance,
there is a relatively large range of values for which the performance is acceptable.
Noting that the case $\tau_A = 0$ corresponds to standard NMF,
we also see that the performance benefit of using MDNMF over NMF increases with model complexity.

\begin{figure}[!htb]
\centering
\includegraphics[width = 0.8\textwidth]{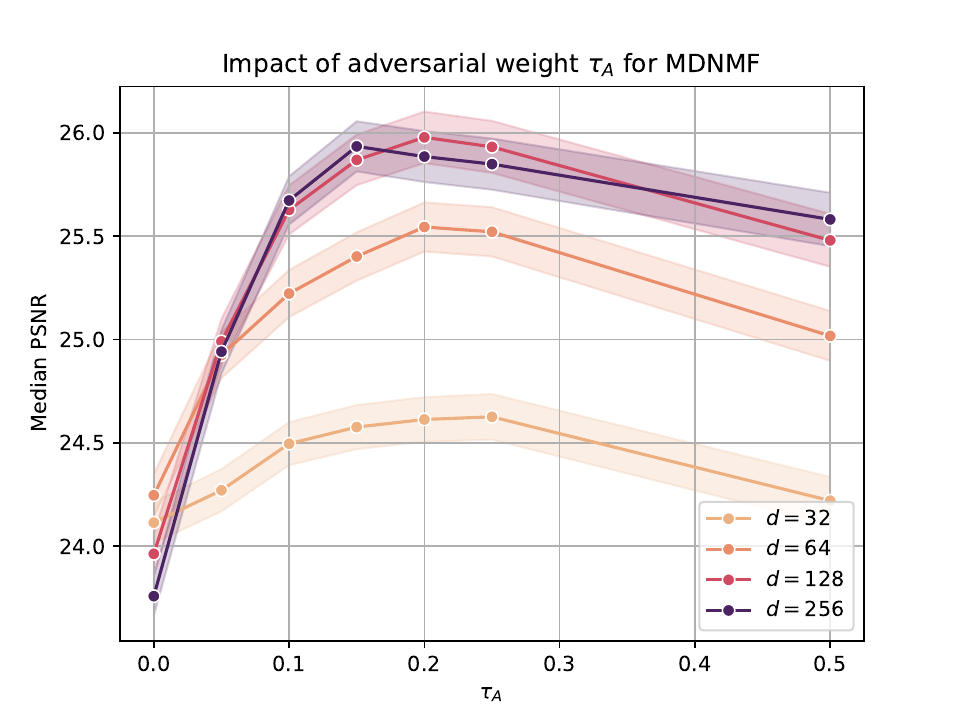}
\caption{Results from experiments in data rich strong supervised setting with $\tau_W = 1$ and varying $\tau_A$ for MDNMF.
The lines show the median PSNR over the test dataset for different parameter values, along with the standard error.
We note that $\tau_A = 0$ corresponds with standard NMF. We see that selecting $\tau_A$ too large leads to much worse performance,
but there is a relatively large range of parameters that yield better performance than standard NMF. We also find that
the discrepancy in performance between NMF and MDNMF becomes larger as $d$ increases.}
\label{fig:data_rich_tau}
\vspace{-\baselineskip} 
\end{figure}

An example of the separation results for a mixed image is shown in Figure~\ref{fig:data_rich_imgs}.
We see that MDNMF and DNMF appear to be better at learning which features belong to the different images.
While standard NMF only learns a set of features that can be used to reconstruct the images, MDNMF and DNMF also learn
what features do not belong to that class of images.
The result of this is that these methods have a smaller tendency to have features of
one source appear in another source, though this comes at the cost of losing some reconstruction accuracy for the relevant data, in particular for outlier data.
Depending on the application, this ability to properly discern what feature belong to which source can be of higher importance than the overall reconstruction quality.

\begin{figure*}[!htb]
\centering
\begin{tabular}{@{\hspace{-1mm}}c@{\hspace{-1mm}}c@{\hspace{-1mm}}c@{\hspace{-1mm}}c@{\hspace{-1mm}}c@{\hspace{-1mm}}}
  \textbf{Mix}
  & \textbf{GT} 
  & \textbf{NMF}  
  & \textbf{MDNMF}
  & \textbf{DNMF}\\
    \includegraphics[width=0.15\textwidth]{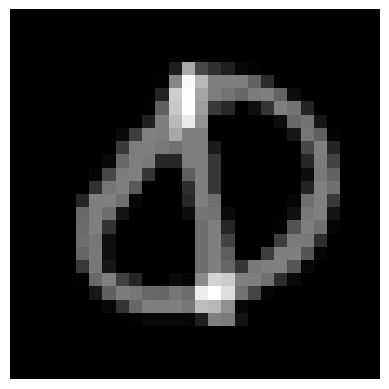}
  & \includegraphics[width=0.15\textwidth]{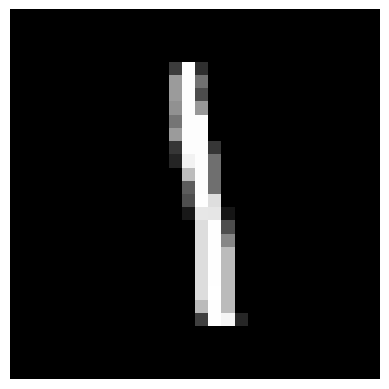}
  & \includegraphics[width=0.15\textwidth]{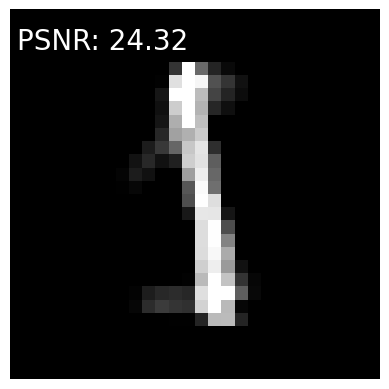}
  & \includegraphics[width=0.15\textwidth]{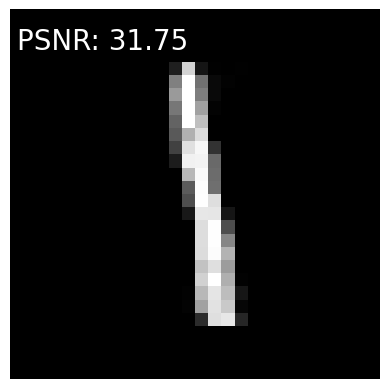}
  & \includegraphics[width=0.15\textwidth]{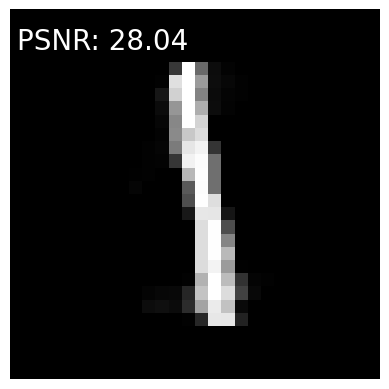} \\
  & \includegraphics[width=0.15\textwidth]{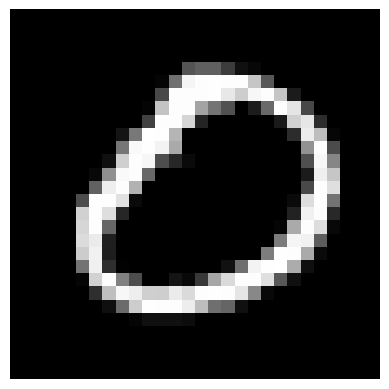} 
  & \includegraphics[width=0.15\textwidth]{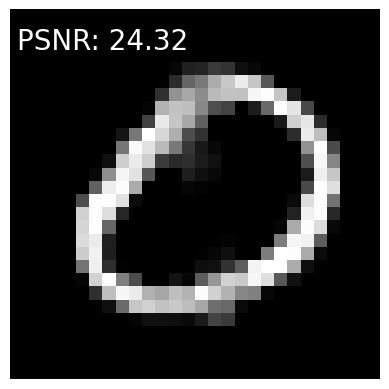} 
  & \includegraphics[width=0.15\textwidth]{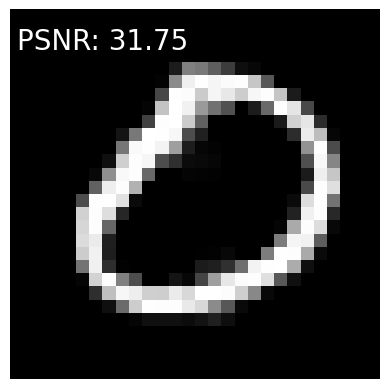} 
  & \includegraphics[width=0.15\textwidth]{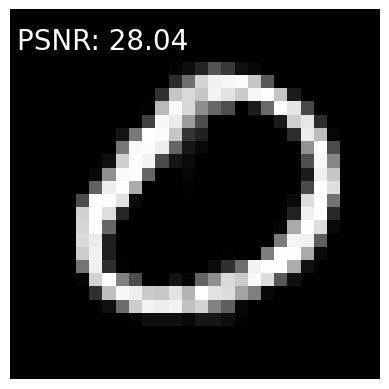} 
\end{tabular}
\centering
\caption{Example separation on test data for bases trained in data rich strong supervised setting with $d = 128$, $\tau_W = 1$, and $\tau_A = 0.2$. 
All images are plotted independently so that the brightest pixel corresponds with the largest pixel value. 
The given PSNR value is the PSNR between the reconstruction and the true source data.
All separations carry some artifacts from the Wiener-filtering around the areas where the mixed images overlap. 
Standard NMF performs especially poorly and the separated images have some clear artifacts. We see this to a much smaller degree
for MDNMF and DNMF, and they are qualitatively similar.}
\label{fig:data_rich_imgs}
\vspace{-\baselineskip} 
\end{figure*}

\subsubsection*{Experiment 2: Sparse data setting}

We now investigate the behavior of the methods in a more realistic, sparser data setting. 
We set the amount of strong supervised data to $N_\text{sup} = 250$
and the amount of weak supervised data for each source to $N_i  = 500$.
This emulates a setting where obtaining strong supervised data is more difficult
than obtaining weak supervision data.
We also generate $N_\text{test} = 1000$ test data, which is not available during training, but will be used to compare the models.
The goal is to investigate how to best utilize this data to fit NMF bases.
We will attempt this by doing hyperparameter tuning as described in section~\ref{sec:hyperparameter}.

The parameters that need to be tuned for the different methods are shown in Table~\ref{tab:tune}.
\begin{table}[h]
\begin{tabular}{c|ccccc}
  \textbf{Parameters}  &  \textbf{NMF} & \textbf{DNMF} & \textbf{MDNMF} & \textbf{D+MDNMF}  \\
  \hline
  \# Basis vectors $d$                          & \cmark & \cmark & \cmark & \cmark \\
  Parameters $\lambda$, $\gamma$                     & \cmark & \cmark & \cmark & \cmark \\
  Test epochs                 & \cmark & \cmark & \cmark & \cmark \\
  Training epochs                      & \cmark & \cmark & \cmark & \cmark \\
  Batch sizes                          & \cmark & \cmark & \cmark & \cmark \\
  Adversarial weight $\tau_A$                             &        &        & \cmark & \cmark \\
  Supervision weight $\tau_S$                             &        &        &        & \cmark
\end{tabular}
\caption{Parameters that need to be tuned for different versions of NMF. Most of these parameters can also be tuned separately for each source. The parameter $\tau_W$ can always be set to $1$.}
\label{tab:tune}
\vspace{-\baselineskip} 
\end{table}

We choose not to tune $d$, as we saw in Figures~\ref{fig:data_rich} and~\ref{fig:data_rich_tau} that 
results tend to improve with $d$ at the cost of computation speed and storage. We ideally
want the basis $W$ to include as few basis vectors as possible, and we choose $d=64$ for all experiments.

We also test for more classes of digits.
Images of ``one'' digits are most suited for NMF based methods, in the sense
that NMF bases trained on this digit do well in the reconstruction.
Therefore, we perform the experiment nine times, each time with a ``one'' digit mixed with 
a different digit. 
We report the difference in median PSNR between the desired
method and standard NMF, denoted $\Delta \text{Median PSNR}$.

The distributions used for the hyperparameters in the random search implementation can be found in the published code. 
For each fit we try $30$ different randomly sampled parameters. The results are shown in Figure~\ref{fig:data_poor}.

\begin{figure}[!htb]
  \centering
  \includegraphics[width = 0.8\textwidth]{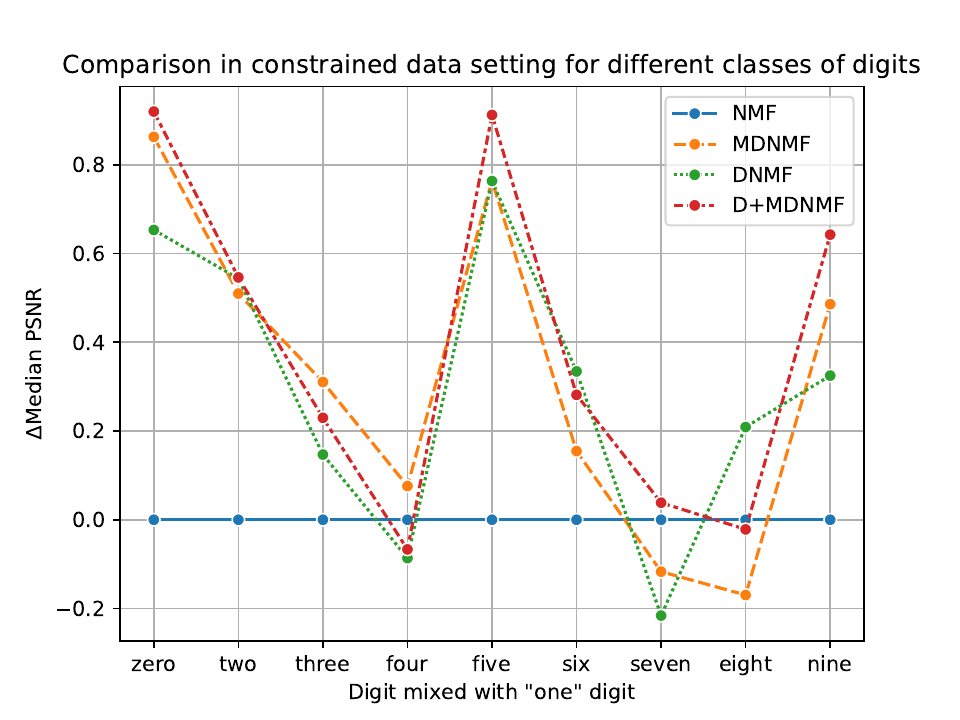}
  \caption{Results from tuning experiments with different digits in the low strong supervision data setting.
  The $y$-axis is the difference in median PSNR between the method and standard NMF.
  The digits on the $x$-axis illustrate which digit was mixed with ``one'' digits when synthetically generating data.
  We observe that there is not a large discrepancy in performance, but D+MDNMF tends to perform best in cases where there is a discrepancy between NMF and the more advanced methods.}
  \label{fig:data_poor}
  \vspace{-\baselineskip} 
\end{figure}

The results seem to indicate that for some digits, or some train--test splits, there is little performance
gain from using more complex methods than standard NMF, and such methods can even overfit the data, yielding worse generalization performance.
However, D+MDNMF tends to perform better, showing that there is some potential benefit to utilizing all data in a careful way.
On the contrary, D+MDNMF also has high variance as it has the most tuning parameters.
Ideally, its performance should always be at least as good as either DNMF or MDNMF, as it is a superset of both methods, but this experiment proves that practically finding optimal parameter values and fitting optimal bases can be challenging.

\subsection{Audio data}

We will now perform a single speaker speech enhancement experiment to further exemplify the usage of 
our proposed methods.
For that, we follow the standard approach for using NMF for audio source separation \cite{fevotte2018single,weninger2014discriminative}.
Also, we only consider the semi-supervised case,
where we have clean speech recordings and noisy speech recordings with different speech from the same speaker, but no clean recordings of noise.
The goal is then to denoise these specific noisy recordings.

\subsubsection*{Data}
In order to make the experiment reproducible, we only use open-source data.
For speech, we use the LibriSpeech dataset \cite{7178964}, which contains $1000$ hours of public domain recordings
of English audiobooks recorded at $16$kHz.  
Specifically, we will use the development set of this dataset, which contains roughly $10$ minutes of speech for each of $38$ different speakers, both male and female.
For noise, we use the Musan dataset \cite{snyder2015musan}, which contains a variety of noise data, including general ambient sounds and technical noises from electronic devices
also recorded at $16$kHz. 
We split the dataset into training and testing, where the training dataset contains half
of the speech data and the test dataset contains half of the speech data with added noise.
We mix the noise additively at a constant SNR of $3$.
Because the noise data is quite varied, the perceived noise level, as well as the difficulty in denoising, varies quite a bit for different types of noise.

\subsubsection*{Feature extraction}

We extract features by first computing the Short-Time Fourier Transform (STFT) and then taking the amplitude
to obtain the spectrum of the audio. We then attempt to separate on the spectrum, and use Wiener-filtering
to reintroduce the phase before applying the Inverse STFT (ISTFT). We note that using the noisy phase for reconstructions
is a simple approach, and more advanced approaches are preferable. However, the point of these experiments are not necessarily to showcase
state-of-the-art speech enhancement, but to test how maximum discrepancy generative regularization performs compared to standard generative regularization for this task.
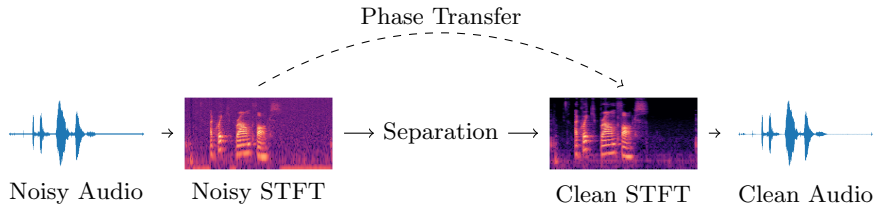
\begin{figure}[!htb]
\begin{tikzpicture}[scale=0.8]
  \node (noisyAudio) at (0,0) {\includegraphics[width=2cm]{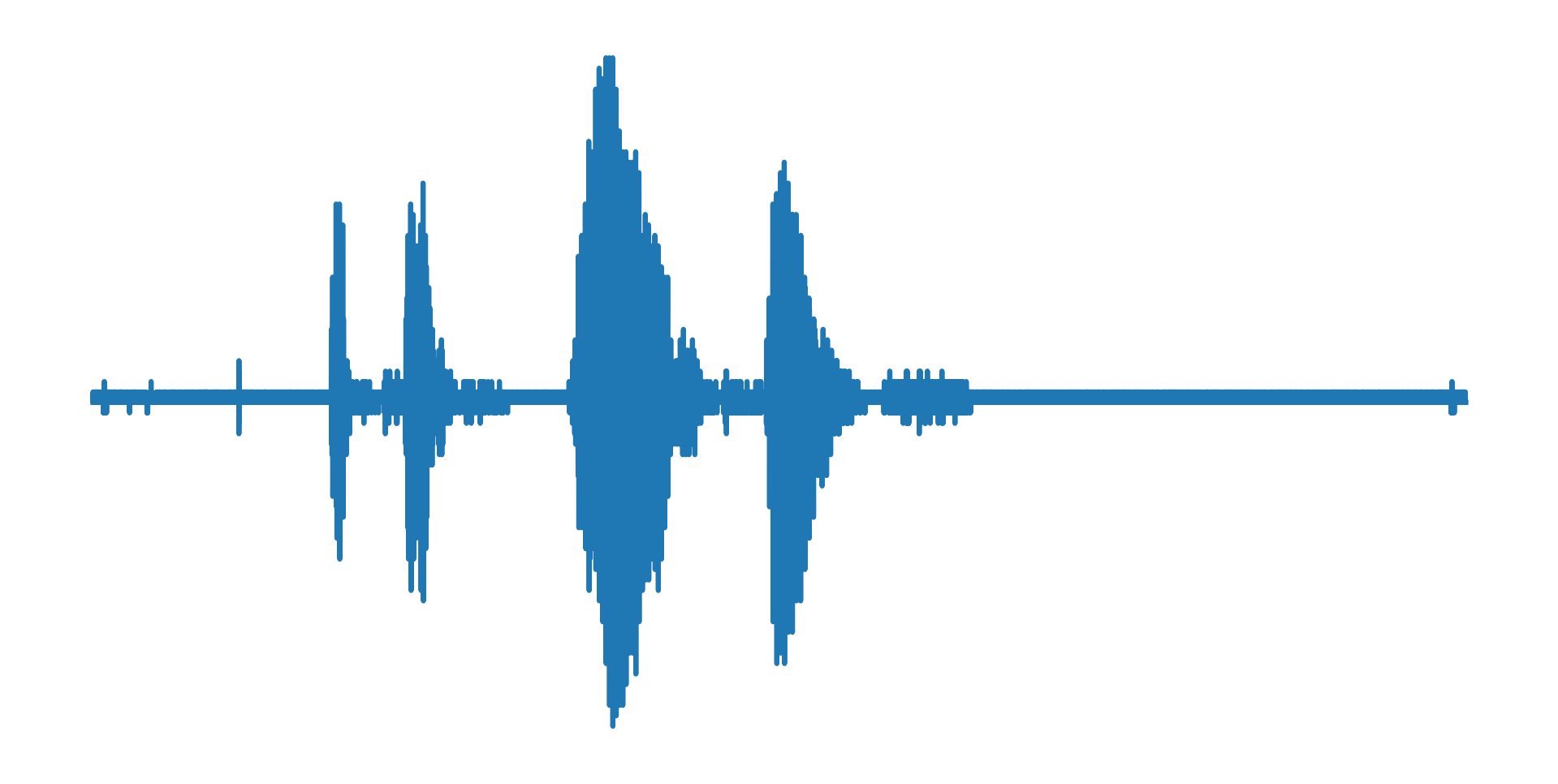}};
  \node (noisyAudioText) at (0,-1.0) {\small Noisy Audio};
  \node (noisySTFT) at (3,0) {\includegraphics[width=2cm]{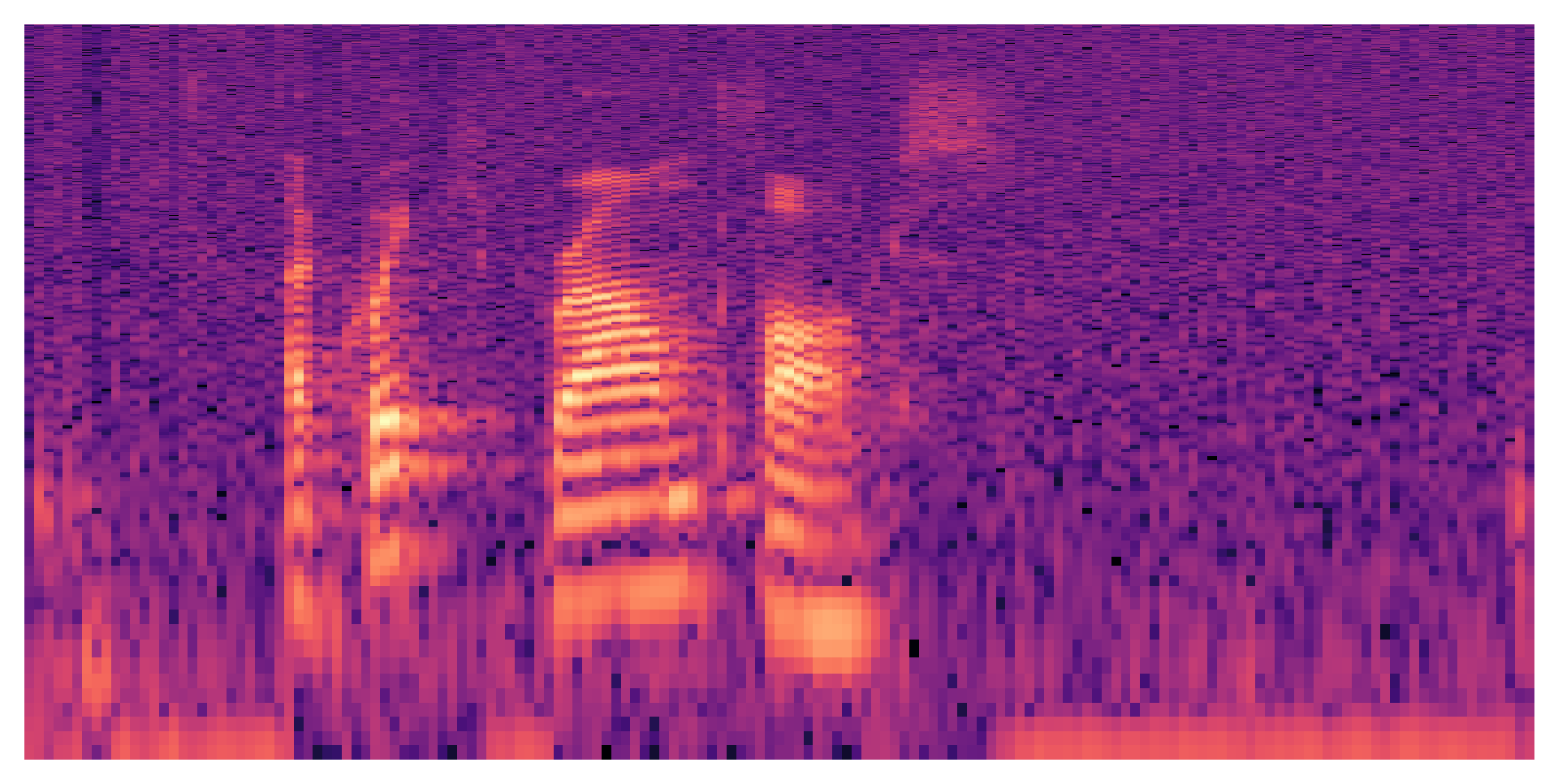}};
  \node (noisySTFTText) at (3,-1.0) {\small Noisy STFT};
  \node (unet) at (6,0) {\small Separation};
  \node (cleanSTFT) at (9,0) {\includegraphics[width=2cm]{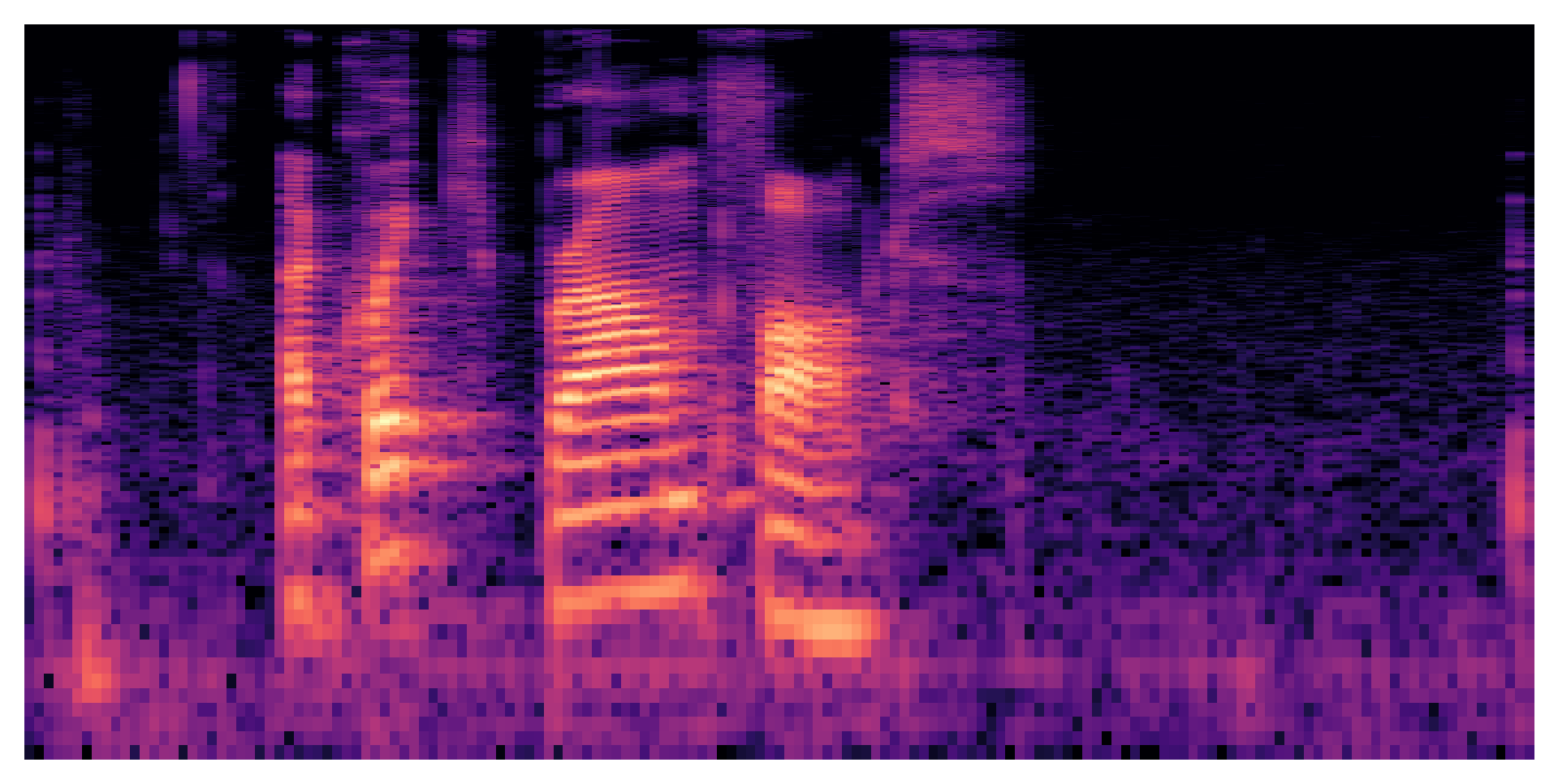}};
  \node (cleanSTFTText) at (9,-1.0) {\small Clean STFT};
  \node (cleanAudio) at (12,0) {\includegraphics[width=2cm]{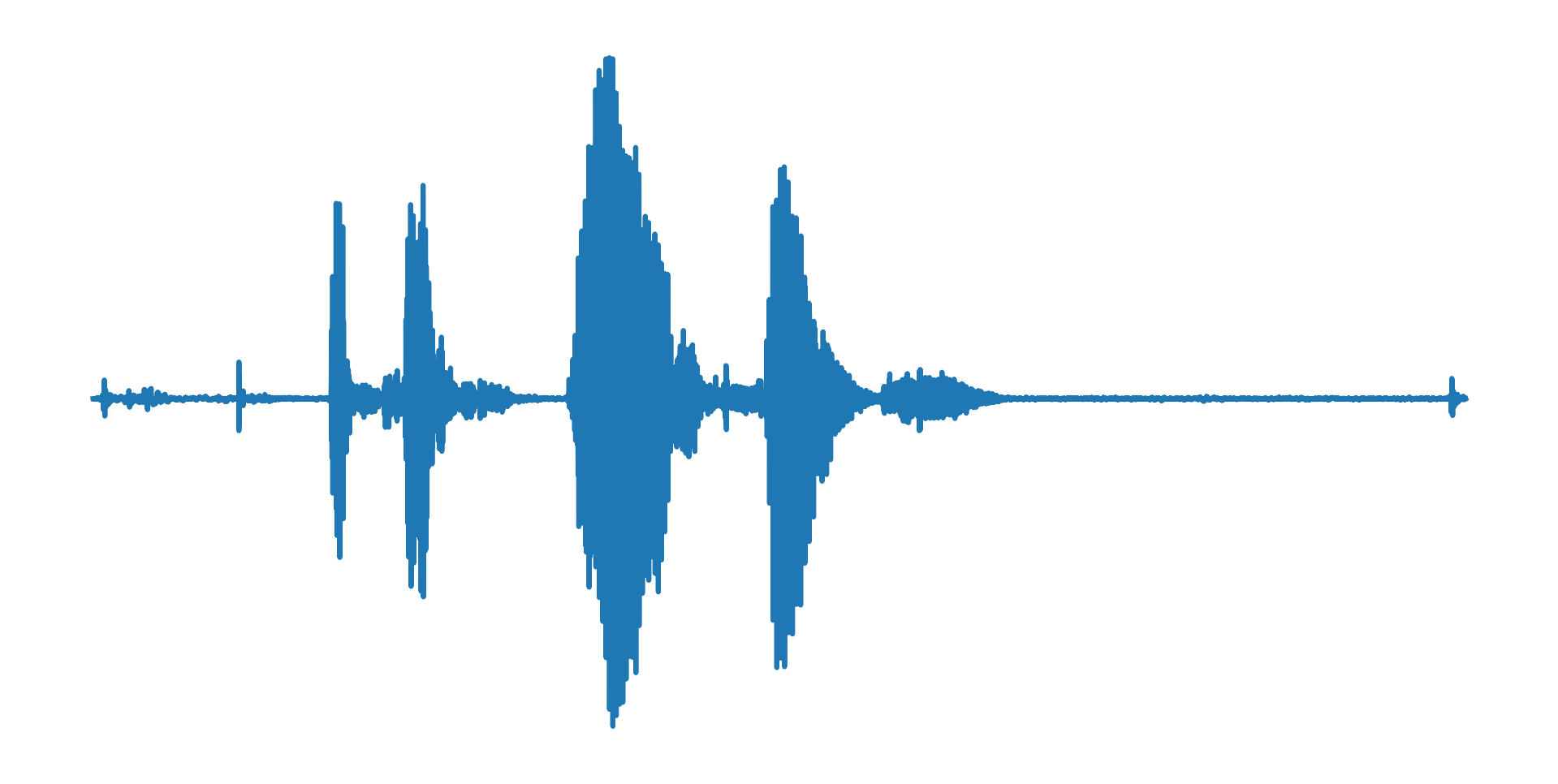}};
  \node (cleanAudioText) at (12,-1.0) {\small Clean Audio};

  \draw[->] (noisyAudio) -- (noisySTFT);
  \draw[->] (cleanSTFT) -- (cleanAudio);
  \draw[->] (unet) -- (cleanSTFT);
  \draw[->] (noisySTFT) -- (unet);

  \draw[->, dashed] (noisySTFT.north) to[bend left] node[midway, above] {\small Phase Transfer} (cleanSTFT.north);
\end{tikzpicture}
\caption{Feature extraction for audio speech using the STFT. The separation is done through the proposed NMF method, and the phase transfer is done by reusing the noisy phase for the reconstructed audio. }
\end{figure}
The librosa package is used to implement the STFT \cite{mcfee2015librosa}. 
In the implementation of the Short-Time Fourier Transform (STFT), a window length and Fast Fourier Transform (FFT) size of 512 samples are utilized, equating to a latency window duration of 32 milliseconds. Furthermore, the hop length is set to ensure a 50\% overlap with adjacent frames.

\subsubsection*{Method}

For speech denoising applications, we are only concerned with reconstructing speech. We first
train a basis for the known speech signal using NMF or MDNMF with the noisy speech as adversarial data, then solve equation \eqref{eq:semi}
to obtain a basis for the unknown noisy signals. We use these bases to separate the noisy signals, and recover the speech signals.
To measure quality, we use SI-SDR \cite{le2019sdr}, which we apply to the individual audio clips. 
We repeat this experiment for different speakers, and do the training from scratch for each speaker.

\subsubsection*{Results}
The results for the audio data can be found at the repository \url{https://github.com/martilud/ANMF}.
For both sparsity and number of basis vectors, we find that selecting different values for the speech and noise is beneficial.
Specifically, we use $d = 128$ basis vectors for speech, and $d = 32$ for noise, as well as $\lambda = 10^{-3}$ for speech and $\lambda = 10^{-10}$ for noise.
The argument for this is that speech is usually complex and sparse, while noise is often not sparse, instead spread over most of the spectrum.
Furthermore, if we allow the noise basis to be too complex, it will simply learn the residual of the NMF representation of speech, though
this is much more problematic for NMF than MDNMF, and MDNMF is much less sensitive to the selection of amount of basis vectors. 
We again see that MDNMF implicitly does model selection, by setting some basis vectors to be roughly equal as $d$ increases.

For MDNMF, we select $\tau_W = 1$ and $\tau_A = 1.0$, and use the approximation described
in Appendix~\ref{app:scaling}. For the calculation of the parameter $\beta_i$
(see~\eqref{eq:beta}) we use the exact weights that were used
for mixing the noisy signals. The results are shown in Figure \ref{fig:audio_results}, where we see that MDNMF significantly outperforms NMF.
We note that in our experiments, while there were some audio clips where NMF quantitatively outperformed MDNMF, MDNMF outperformed NMF on average for all $38$ speakers.
We also show the spectrogram of the result of a specific audio clip in Figure \ref{fig:audio_example}.
\begin{figure}[!htb]
  \includegraphics[width = 0.9\textwidth]{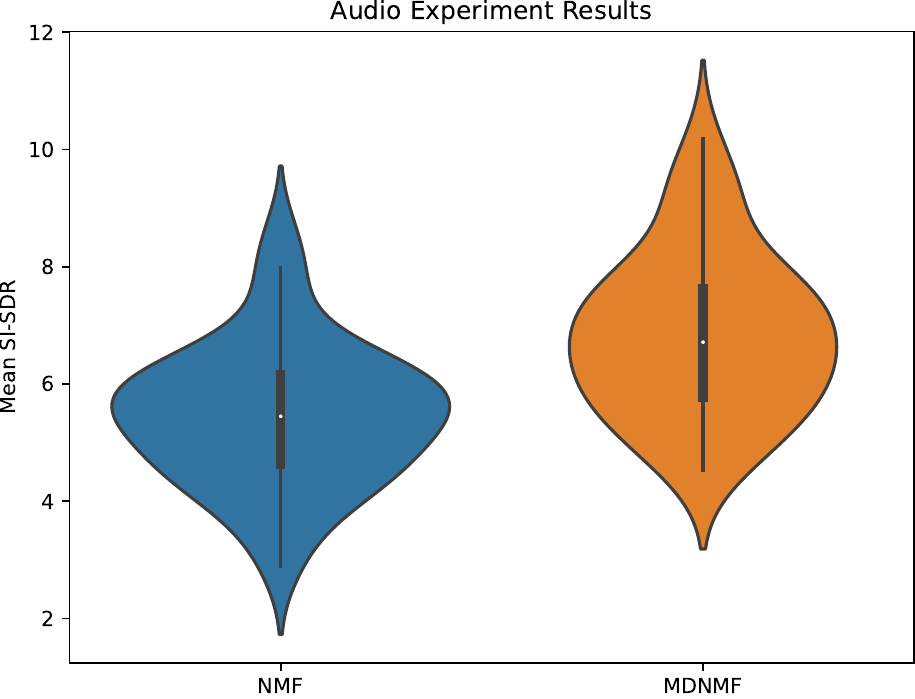}
  \caption{Results from audio experiment. We see that MDNMF consistently outperforms NMF in terms of mean SI-SDR of the different experiments for different speakers. }
  \label{fig:audio_results}
\end{figure}

\begin{figure}[!htb]
  \includegraphics[width = 0.49\textwidth]{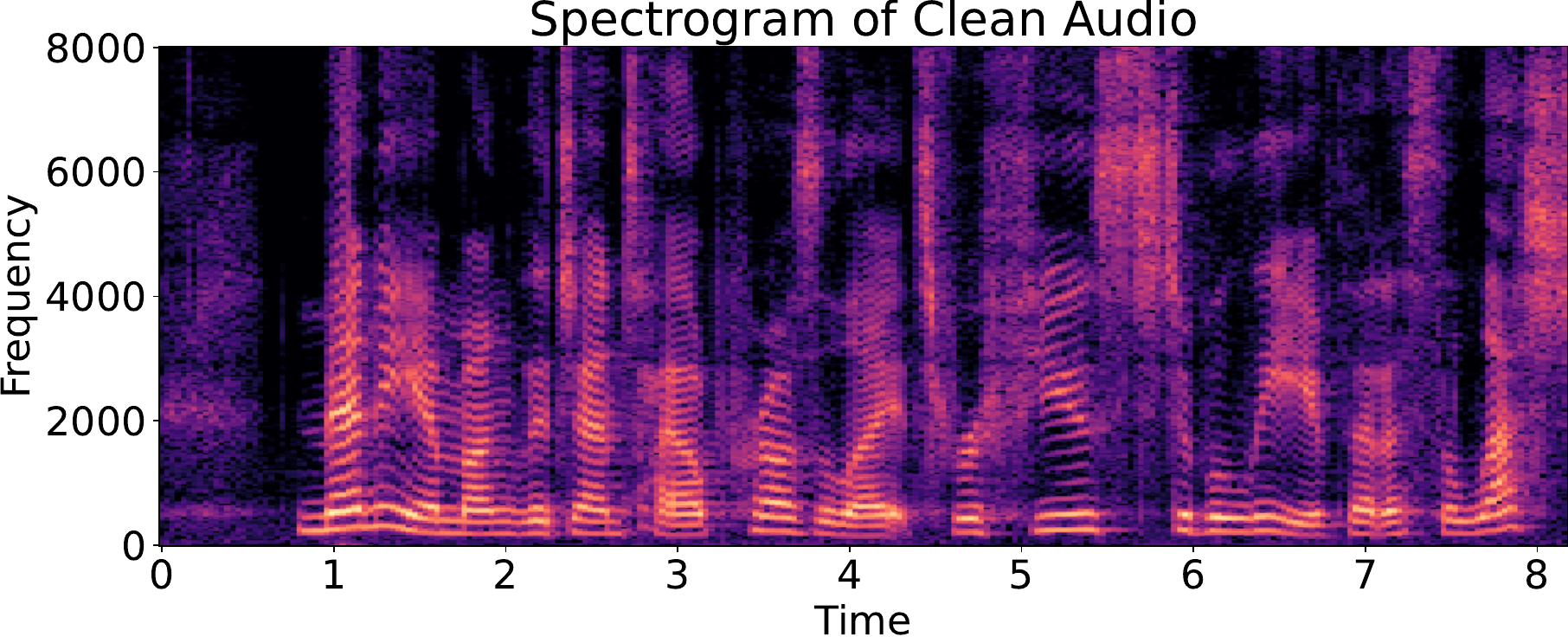}
  \includegraphics[width = 0.49\textwidth]{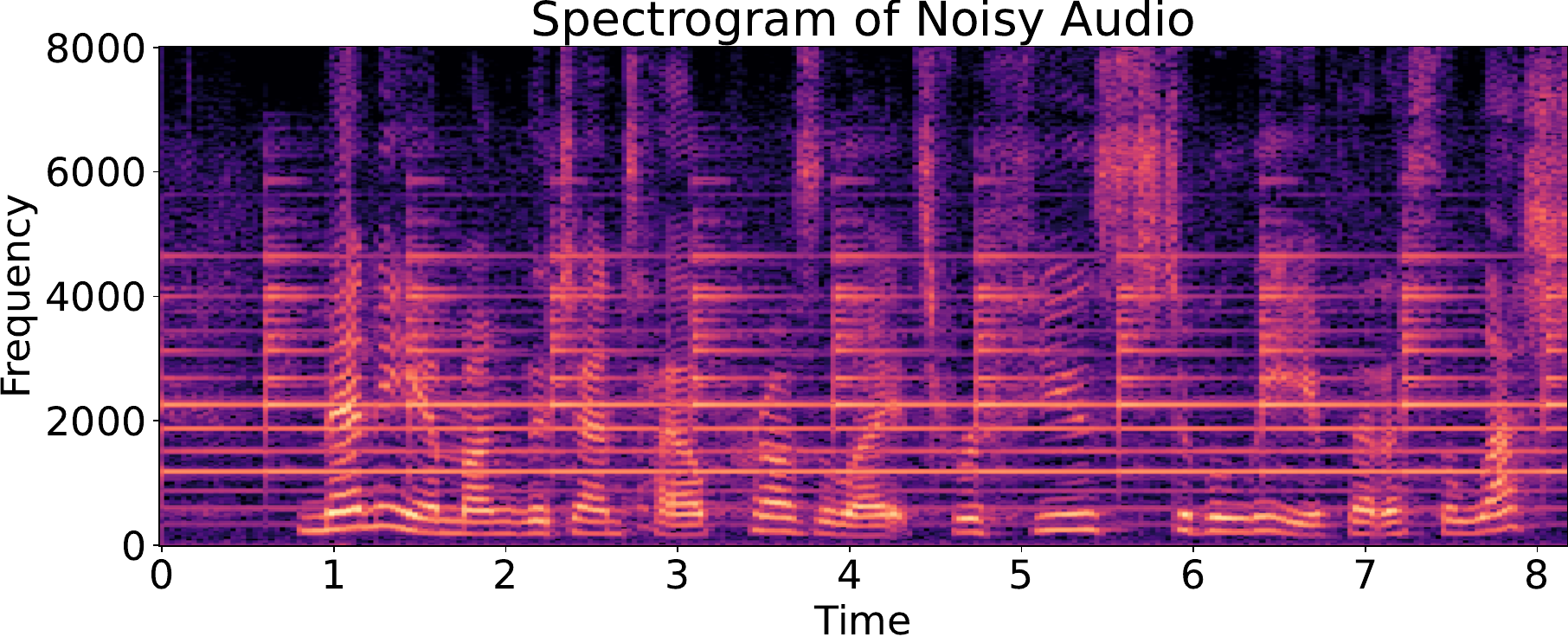}
  \includegraphics[width = 0.49\textwidth]{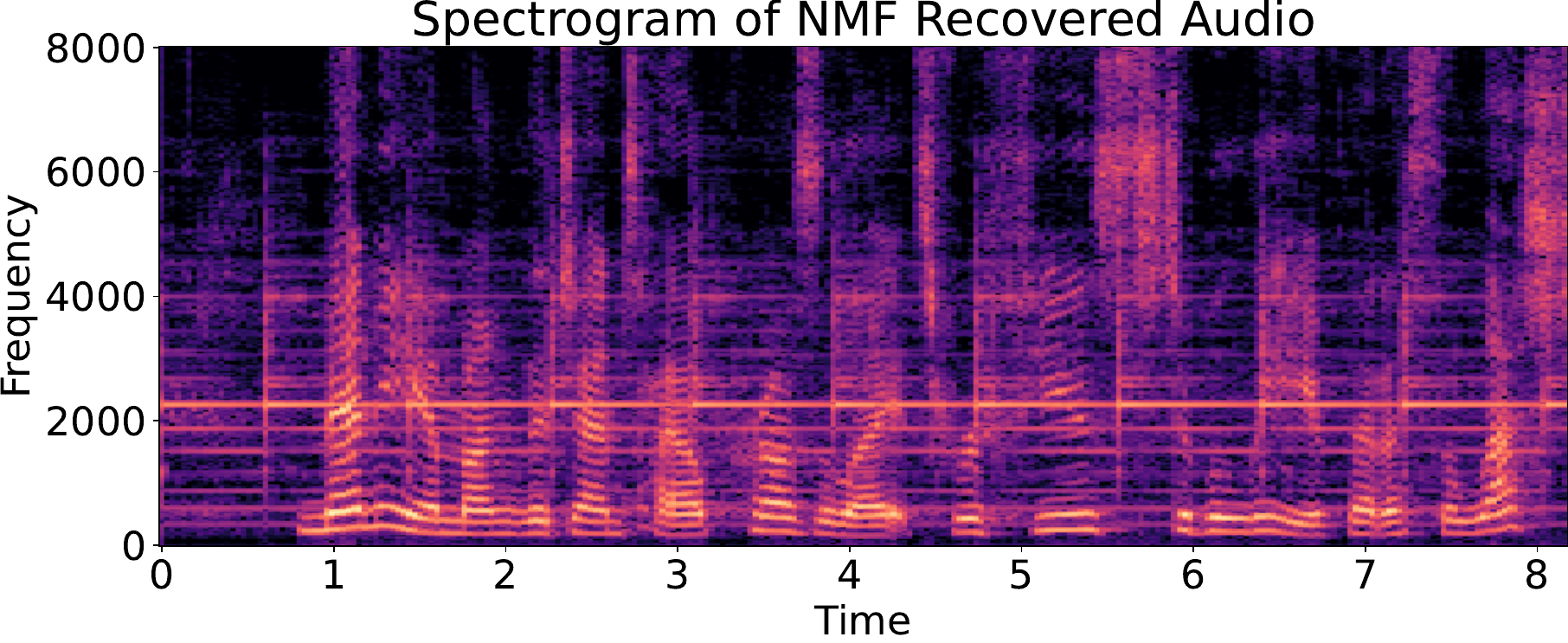}
  \includegraphics[width = 0.49\textwidth]{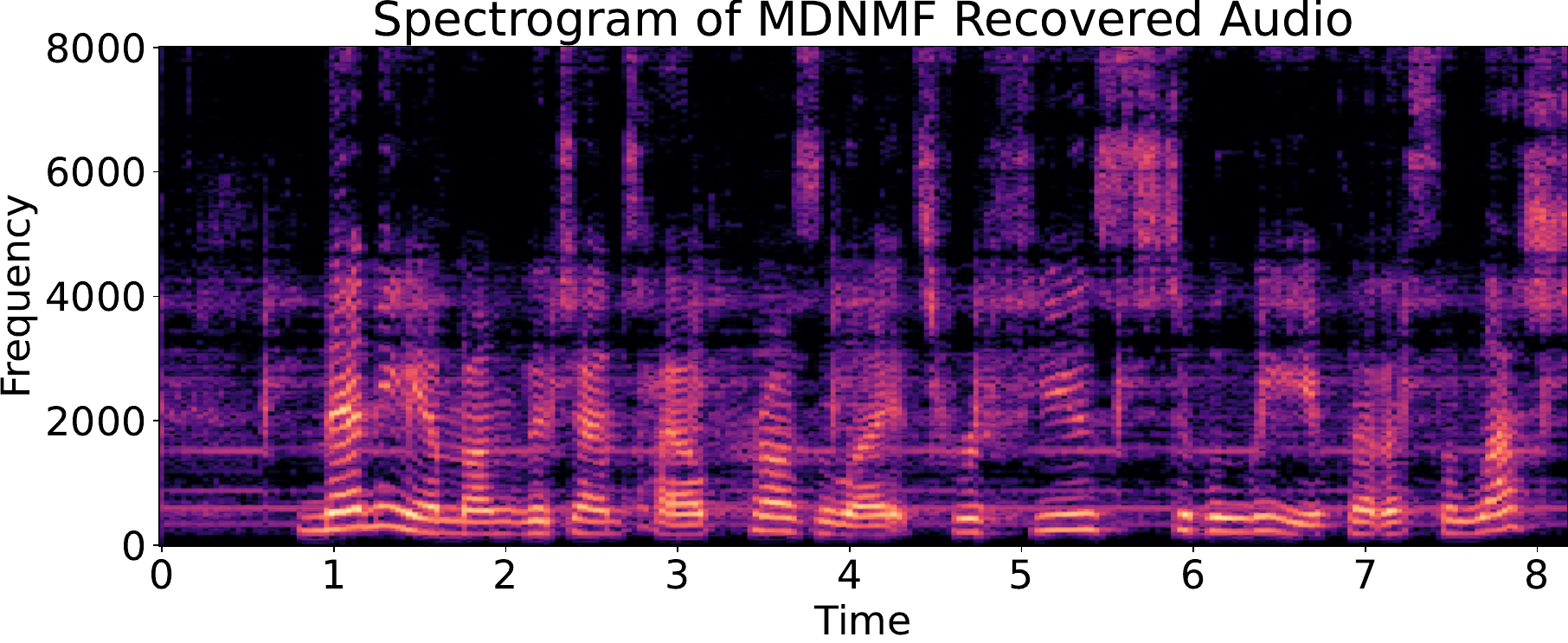}
  \caption{Example clean, noisy, NMF recovered and MDNMF recovered audio. MDNMF is clearly more capable at removing noise, though there is still some room for improvement. In particular, for this specific audio clip
  it is clear that the noise also exhibits sparsity.}
  \label{fig:audio_example}
\end{figure}

Qualitatively we note that the denoised signals produced by MDNMF are of higher or equal quality, and there is still much room for parameter tuning. A feature of MDNMF is that it consistently
removes more stationary (low-frequency) noise from the signal, as well as specific frequencies that are not required for reconstruction of speech signals.
Another feature of MDNMF is that there is a trade-off between how much noise is removed and the quality of reconstructed speech.
This means that some reconstructions where MDNMF performs quantitatively worse often remove more of the noise. 
As such, MDNMF is preferable not only because of the superior quantitative performance, but especially in situations where removal of noise
is more important than perfect reconstruction of the speech signal.
We also see that both methods perform worse at removing non-stationary noise, like background music or sharp sounds from moving objects.
We further note that using the noisy phase is clearly insufficient, and more complex phase transfer methods should be used to obtain high quality reconstructions.
In order to be able to treat this type of noise with NMF-based methods, we suspect one would need more basis vectors to fully model the complexity of the noise, as well as more data to properly fit the bases.
Another approach would be to use generative models that are better capable at modeling non-stationary noise data.

The conclusion is that MDNMF can learn compact bases that can be used for single speaker speech denoising in a semi-supervised data setting. 
We believe MDNMF would also be suitable for more large scale multi-speaker denoising and other audio applications, but this would require further investigation.

\section{Further Work}
Several potential avenues for future research are worth exploring,
including investigating maximum discrepancy generative regularization for more complex generative methods, incorporating transfer learning techniques,
and identifying better approaches for parameter tuning.
We believe that maximum discrepancy generative regularization can be a valuable tool for source separation and more general inverse problems in weak supervision data settings.
The main challenge is how to most efficiently utilize trained models and data available to achieve good and robust results.

\section{Conclusion}
In this article we have investigated maximum discrepancy generative regularization for single channel source separation. In particular, we have developed a variant of Non-Negative 
Matrix Factorization that we called Maximum Discrepancy Non-Negative Matrix Factorization (MDNMF). 
We have discussed how to utilize weak and strong supervision data for training adversarial generative functions for source separation problems.
We have seen in the numerical
experiments that MDNMF outperforms other NMF methods 
for both image and audio source separation problems,
including methods that make use of strong supervision data.

\appendix

\section{Proof of Theorem~\ref{theorem:existence}}\label{app:proofs_1}

Since the sets $\mathcal{G}_i$ are sequentially compact,
it is sufficient to show that the mappings $\mathcal{F}_i^W$ and $\mathcal{F}^A_i$ are continuous
and that the mappings $\mathcal{F}_i^S$ are lower semi-continuous.
For simplicity of notation, we drop in the following the indices $i$.

We start with the continuity of the mapping $\mathcal{F}^W$.
For this, we show first that the mapping
$g \mapsto \lVert \pi(g,y)-u\rVert^2$ is continuous.
To that end, denote
\[
  P(h,g) = \lVert g(h)-u\rVert^2 + \lambda \mathcal{R}(h).
\]
Because the mapping $(g,h) \mapsto g(h)$ is continuous and $\mathcal{R}$
is lower semi-continuous, it follows that $P$ is lower semi-continuous.

Assume now that the sequence $\{g^{(k)}\}_{k\in\N} \subset \mathcal{G}$ converges to $g_0 \in \mathcal{G}$
and let $h^{(k)}$ be minimizers of the functions $P(\cdot,g^{(k)})$.
Let moreover $\hat{h} \in \R^{d}$ be such that $\mathcal{R}(\hat{h}) < \infty$.
Then
\begin{multline*}
  \lambda\mathcal{R}(h^{(k)}) \le P(h^{(k)},g^{(k)}) \le P(\hat{h},g^{(k)}) = \lVert g^{(k)}(\hat{h})-u\rVert^2 + \lambda\mathcal{R}(\hat{h})\\
  \le \sup_{g \in \mathcal{G}} \lVert g(\hat{h})-u\rVert^2 + \lambda\mathcal{R}(\hat{h})
\end{multline*}
for all $k$.
Because $\mathcal{G}$ is sequentially compact and the mapping $g \mapsto g(\hat{h})$ is continuous,
it follows that
\[
  \sup_{g \in \mathcal{G}} \lVert g(\hat{h})-u\rVert^2 < \infty.
\]
Because $\mathcal{R}$ is coercive, it follows that the sequence $\{h^{(k)}\}_{k\in\N} \subset \R^d$
is bounded.
After possibly passing to a subsequence we may thus assume that
it converges to some $h_0 \in \R^d$.

Since $P$ is lower semi-continuous and $h^{(k)}$ minimizes
$P(\cdot,g^{(k)})$, it follows that
\[
  P(h_0,g_0) \le \liminf_{k\to \infty} P(h^{(k)},g^{(k)}) \le \liminf_{k\to\infty} P(h,g^{(k)})
\]
for all $h \in \R^d$.
Since $P$ is continuous with respect to $g$ and $g_0 = \lim_{k\to\infty} g^{(k)}$,
it follows that $P(h_0,g_0) \le P(h,g_0)$ for all $h$, which implies that
$h_0$ is a minimizer of $P$.
The continuity of $g$ now implies that
\[
  \lim_{k\to\infty} \lVert \pi(g^{(k)},u) -u\rVert^2 =
  \lim_{k\to\infty} \lVert g^{(k)}(h^{(k)}) - u \rVert^2
  = \lVert \pi(g_0,u)-u\rVert^2.
\]
Finally, Assumption~\ref{ass:2} implies that the value
of $\lVert\pi(g_0,u)-u\rVert^2$ does not depend on the choice
of the convergent subsequence chosen earlier.
Thus a standard subsequence argument implies the
(sequential) continuity of the mapping
$g \mapsto \lVert \pi(g,y)-u\rVert^2$.

Next we obtain that
\begin{multline*}
  \lVert\pi(g^{(k)},u)-u\rVert^2 \le
  P(h^{(k)},g^{(k)}) = \lVert g(\hat{h})-u\rVert^2 + \lambda\mathcal{R}(\hat{h})\\
  \le 2 \lVert g^{(k)}(\hat{h}) \rVert^2 + 2\lVert u \rVert^2 + \lambda\mathcal{R}(\hat{h})
  \le 2 \sup_{g\in\mathcal{G}} \lVert g(\hat{h}) \rVert^2 + \lambda\mathcal{R}(\hat{h}) + 2\lVert u \rVert^2
\end{multline*}
for all $k \in \N$.
Since $\mathbb{E}_{u\sim\mathbb{P}_U} (\lVert u \rVert^2) < \infty$,
we can thus apply Lebesgue's Theorem of Dominated Convergence
and obtain that
\[
  \lim_{k\to\infty} \mathbb{E}_{u\sim \mathbb{P}_U} (\lVert\pi(g^{(k)},u)-u\rVert^2)
  = \mathbb{E}_{u\sim\mathbb{P}_U}(\lVert \pi(g_0,u) - u\rVert^2),
\]
which proves the continuity of $\mathcal{F}^W$.
\medskip

The proof of the continuity of $\mathcal{F}^A$ is similar.
The only difference is that we require for the application of Lebesgue's Theorem
of Dominated Convergence that $\mathcal{E}_{u\sim \mathbb{P}_Z}(\lVert u \rVert^2) < \infty$.
This, however, holds because of the construction of the distribution $\mathbb{P}_Z$
in view of the boundedness of the coefficients $a_i$ and the finiteness
of the second moments of the data.
\medskip

Finally, we show the lower semi-continuity of the term $\mathcal{F}^S$.
Here we start, similarly as above, by showing that the mapping $\mathbf{g} \mapsto \lVert u^*(\mathbf{g};v) - u\rVert^2$
is lower semi-continuous.

Assume therefore that the sequence $\{\mathbf{g}^{(k)}\}_{k\in\N}$ converges
to $\mathbf{g}_0$. Moreover, let $(\mathbf{a}^{(k)},\mathbf{u}^{(k)},h^{(k)})$
be the corresponding solutions of~\eqref{eq:TikGen}.
As in the argumentation above, we then obtain that the sequence $\{\mathbf{h}^{(k)}\}_{k\in\N}$
is bounded, and thus we can choose a subsequence converging to some
$\mathbf{h}_0$.
By the continuity of the evaluation mapping,
the corresponding sequence
of generated data $\mathbf{u}^{(k)} = \mathbf{g}^{(k)}(\mathbf{h}^{(k)})$ converges to
$\mathbf{u}_0 = \mathbf{g}_0(\mathbf{h}_0)$.
In addition, the parameters $\mathbf{a}^{(k)}$ are by definition uniformly bounded
and thus also admit a subsequence converging to some $\mathbf{a}_0$.
In view of the lower semi-continuity of the functional in~\eqref{eq:TikGen},
it then follows that
$(\mathbf{a}_0,\mathbf{u}_0,\mathbf{h}_0)$ is a solution of~\eqref{eq:TikGen}
for the generator $\mathbf{g}_0$.

Now note that $u^*(\mathbf{g}_0;v)$ was chosen to be
the $i$-th component of the solution of~\eqref{eq:TikGen} for which
$\lVert u^*(\mathbf{g}_0;v) - u\rVert$ was minimal.
Thus
\[
  \lVert u^*(\mathbf{g}_0;v) - u\rVert^2
  \le \lVert u_0 - u\rVert^2
  = \lim_{k\to\infty} \lVert u^{(k)} - u \rVert^2.
\]
Since this holds independent of the choice of the convergent subsequence,
it follows that the mapping
$\mathbf{g} \mapsto \lVert u^*(\mathbf{g};v) - u\rVert^2$ is
lower semi-continuous.
Now the lower semi-continuity of $\mathcal{F}^S$ is an immediate consequence
of Fatou's lemma.  

\section{Proof of Theorem~\ref{theorem:nonincrease}}\label{app:proofs_2}

The argument follows the original proof in \cite{lee2000algorithms}.
We are interested in the loss
\begin{equation}
  L(w) = \frac{1}{2}\|u - wH\|^2 - \frac{1}{2}\|\hat{u} - w\hat{H}\|^2 + \frac{1}{2}\|\tilde{u} - w\tilde{H}\|^2 + \gamma \|w\|_1,
  \label{eq:loss}
\end{equation}
where $w \in \mathbb{R}_+^{d}$ is a row vector, $u, \hat{u}, \tilde{u} \in \mathbb{R}_+^{n}$ and $H, \hat{H}, \tilde{H} \in \mathbb{R}_+^{d \times n}$.
We assume that all data is strictly positive. 
For simplicity we assume that the amount of weak supervision, adversarial and strong supervision data is equal.

The D+MDNMF update in this case becomes
\begin{equation}
  w^{k+1} = w^k \odot \frac{uH^T + \tilde{u}\tilde{H}^T + w^k \hat{H}\hat{H}^T}{w^k(HH^T + \tilde{H}\tilde{H}^T) + \hat{u}\hat{H}^T + \gamma}.
  \label{eq:update}
\end{equation}

The main tool we need for the proof are so-called auxiliary functions.
\begin{definition}
  An auxiliary function for $L(w)$
  is a function $M \colon \mathbb{R}^d \times \mathbb{R}^d \to \mathbb{R}$
  such that
  \[L(w) \le M(w,w^k), \quad L(w) = M(w, w), \]
  for all $w$, $w^k \in \mathbb{R}^d$.
\end{definition}
\begin{lemma}\label{le:auxiliary}
  If $M(w,w^k)$ is an auxiliary function of $L(W^k)$, then given the update
\[w^{k+1} = \arg\min_{w} M(w,w^k), \]
  we have
  \[L(w^{k+1})\le L(w^k).\]
\end{lemma}
\begin{proof}
By the definition of auxiliary functions 
\[L(w^{k+1}) \le M(w^{k+1}, w^k) \le M(w^k,w^k) = L(w^k), \]
where the second inequality follows from the fact that $w^{k+1}$ is a minimizer of $M(w,w^k)$. 
\end{proof}

In the following, we will show that the multiplicative update~\eqref{eq:DANMFW}
is of the form described in Lemma~\ref{le:auxiliary} with the auxiliary function 
\[
  M(w, w^k) = L(w^k) + (w - w^k)^T \nabla_w L(w^k) + (w - w^k)^TK(w^k)(w - w^k)
\]
for the loss~\eqref{eq:loss}, where 
\[
  \nabla_w L(w^k) = uH^T + \hat{u}\hat{H}^T + \tilde{u}\tilde{H}^T - w(H^T + \hat{H}^T + \tilde{H}^T) + \gamma 
\]
and
\[
  K(w^k) = \text{diag}\left(\frac{w_k(HH^T + \tilde{H}\tilde{H}^T) + \hat{u}\hat{H}^T + \gamma}{w_k}\right).
\]
Here the division is interpreted componentwise, 
As this function is quadratic and strictly convex it is clear that
\[ w^{k+1} = w^{k} - K(w^{k})^{-1}\nabla_w L(w^k)\]
is a minimizer of $M(w,w^k)$. 
Moreover, one can verify that this update leads to the proposed multiplicative update~\eqref{eq:update}.

Thus, if we can show that $M(w,w^k)$ is an auxiliary function for the loss, it follows that the loss 
is non-increasing under this update.
Following \cite{lee2000algorithms} this is equivalent to showing that the matrix
\[
  B(w^k) = K(w^k) - H^TH + \hat{H}^T\hat{H} - \tilde{H}^T\tilde{H},
\]
which is the difference between the proposed auxiliary function and the expansion of the loss~\eqref{eq:loss} around $w^k$, is positive semi-definite. 
The terms corresponding to weak supervision data and by extension strong supervision data
have been shown to be positive semi-definite by Lee and Seung, and the remaining terms are also positive semi-definite as
\[
  w^T(\text{diag}(\frac{\hat{u}\hat{H}^T + \gamma}{w^k}) + \hat{H}\hat{H}^T)w \ge \|\hat{H}^Tw\|^2 \ge 0,
\]
given the non-negativity of $\hat{u} \hat{H}^T$ and $\gamma$. 
It is worth noting that when including the adversarial term, $B(w^k)$ is strictly positive definite when $\hat{u}$ and $\hat{H}^T$ are strictly positive. 

Thus, $M(w,w^k)$ is an auxiliary function for the proposed loss~\eqref{eq:loss}, and thus the loss is non-increasing under the update. It follows that the full D+MDNMF is non-increasing under the update when the latent variables are fixed.
\begin{remark}
  If any component of $w^k$ is $0$, it will stay $0$ under the update. This can only happen analytically if it is initialized as such or if the data is not strictly positive, and numerically due to rounding.
  In this case, we set the corresponding component of $K(w^k)^{-1}$ to be $0$, which again corresponds to the proposed update, and the non-increase of the loss~\eqref{eq:loss} still holds.
\end{remark}
\section{Scaling of Adversarial Data}
\label{app:scaling}
Generally, we can select adversarial data for the $i$-th source following the mixture distribution 
\[
  \mathbb{P}_{Z_i} := \sum_{j\neq i} \omega_{ij} \mathbb{P}_{{U}_j}
  + \omega_{ii}\mathbb{P}_{V_i},
\]
where $\omega_{ij}$ represents the weight of the corresponding distribution of data $\mathbb{P}_{{U}_j}$, and satisfies the constraints $0 \le \omega_{ij} \le 1$ and $\sum_{j} \omega_{ij} = 1$.
Moreover, $\mathbb{P}_{V_i} = (f_i)_{\#}(\mathbb{P}_{\mathcal{A}\times V})$,
with $f_i$ as defined in~\eqref{eq:fi}.
A natural choice for the weights is to select $\omega_{ij}$ equal to the ratio
of the amount of available data from the corresponding distribution and the total amount of adversarial data. However, more nuanced selection of 
these weights might be interesting, for example when training adversarially against one class is more important than training adversarially against another class.

A particularly useful property of NMF is that for any $\alpha >0$, we have
\begin{align*} 
  h(W,\alpha u;\alpha\lambda) &= \argmin_{h \ge 0} \|\alpha u - Wh\|^2 +\alpha\lambda \lvert h\rvert_1 \\ 
  &= \argmin_{h \ge 0} \alpha^2 \Bigl( \Bigl\lVert u - W\frac{h}{\alpha}\Bigr\rVert^2 + \lambda \Bigl\lvert\frac{h}{\alpha}\Bigl\rvert_1 \Bigr) \\
  &= \alpha h(W, u; \lambda).
\end{align*}
Thus, we can compute the expected value over the adversarial data as
\begin{multline*}
  \mathbb{E}_{u \sim \mathbb{P}_{Z_i}}[\|u - Wh(W,u;\lambda)\|^2] = \sum_{j \neq i}\mathbb{E}_{u \sim \mathbb{P}_{U_j}}[\|\sqrt{\omega_{ij}}u - Wh(W,\sqrt{\omega_{ij}}u;\sqrt{\omega_{ij}}\lambda)\|^2] \\
  + \mathbb{E}_{u \sim \mathbb{P}_{V_i}}[\|\sqrt{\omega_{ii}}u - Wh(W,\sqrt{\omega_{ii}}u;\sqrt{\omega_{ii}}\lambda)\|^2].
\end{multline*}
We can then store the adversarial data in a matrix
\[ \hat{U}_i = \begin{bmatrix}\sqrt{\omega_{i1}}U_1 & \hdots  & \sqrt{\omega_{i(i-1)}} U_{i-1}& \sqrt{\omega_{i(i+1)}} U_{i+1} & \hdots & \sqrt{\omega_{iS}} U_S & \sqrt{\omega_{ii}} V_i \end{bmatrix} \]
and obtain with Monte Carlo integration the approximation
\[ \mathbb{E}_{u \sim \mathbb{P}_{Z_i}}[\|u - W_ih(W_i,u;\lambda)\|^2] \approx \frac{1}{\hat{N}_i}\|\hat{U}_i - W_iH(W_i,\hat{U}_i)\|_F^2,\]
where the sparsity parameter for $H(W, \hat{U}_i)$ needs to be scaled per data according to $\omega$. This 
can easily be done with the proposed numerical update.  
Scaling the sparsity parameter only has an effect if the sparsity parameter $\lambda$ is large, and
can likely be ignored in some applications.

Taking the expected value over $\mathbb{P}_{V_i}$ requires samples from the mixed data
and the weights in order to approximate the joint distribution
$\mathbb{P}_{\mathcal{A}\times V}$.
For the particular case of NMF, however, we can obtain a reasonable approximation
that only requires sampling the data $V$:
In the cases where $\alpha \lambda \ll 1$ or $\alpha \approx 1$ we have 
\[
 \alpha h(W,u;\lambda) \approx h(W,\alpha u; \lambda),  
\]
which further implies that
\[
  \lVert \alpha u - W_i h(W_i,\alpha u;\lambda)\rVert_2^2 \approx \alpha^2 \lVert u- W_i h(W_i,u;\lambda)\rVert_2^2.
\]
Because of the assumed independence of $\mathbb{P}_V$ and $\mathbb{P}_{\mathcal{A}}$, we further have that
\begin{align*}
  \mathbb{E}_{u \sim \mathbb{P}_{{V}_i}}\bigl(\lVert u - W_i h(W_i,u)\rVert^2\bigr)
  &= \mathbb{E}_{(a,u) \sim \mathbb{P}_{\mathcal{A} \times V}}\Bigl(\Bigl\lVert \frac{a_i}{\sum_{j} a_j^2}u - W_i h\Bigl(W_i,\frac{a_i}{\sum_{j} a_j^2}u\Bigr)\Bigr\rVert^2\Bigr)\\
  &\approx\mathbb{E}_{(a,u) \sim \mathbb{P}_{\mathcal{A} \times V}}\Bigl(\Bigl(\frac{a_i}{\sum_{j} a_j^2}\Bigr)^2 \bigl\lVert u - W_i h(W_i,u)\bigr\rVert^2\Bigr)\\
  &= \beta_i \mathbb{E}_{u \sim \mathbb{P}_{V}} \bigl(\lVert u - W_i h(W_i,u)\rVert^2\bigr)
\end{align*}
with
\begin{equation}
  \beta_i = \mathbb{E}_{a \sim \mathbb{P}_A}\Bigl(\Bigl(\frac{a_i}{\sum_{j} a_j^2}\Bigr)^2\Bigr).
  \label{eq:beta}
\end{equation}

We note that $a_i/(\sum_{j} a_j^2) \approx 1$ is the case when $a_i \approx 1/S$, that is, when all weights are roughly the same size.

\section{Additional notes on numerical implementation}
\label{app:numerical}

When applying stochastic multiplicative updates to ANMF and D+MDNMF we face the challenge that we are minimizing a loss with
different terms with potentially unbalanced data. This means that we can, and most likely should, select different batch sizes for the different datasets.
Some potential strategies for this selection are the following:
\begin{itemize}
    \item \textbf{Proportional sampling}: Select the batch sizes proportional to the size of the datasets, in which case the number of batches are also equal. 
    \item \textbf{Undersampling}: End the epoch and reshuffle all data when we have passed through all batches of the dataset with the smallest number of batches.
    \item \textbf{Oversampling}: When all batches of one dataset have been passed through, resample the data-points from this dataset until all batches of the other datasets have been passed through.
    \item \textbf{Iterative sampling}: Reshuffle only the corresponding data when we have passed through all batches of this dataset. In this case each dataset has its corresponding number of epochs.
\end{itemize}
All approaches have their advantages and disadvantages, but we propose using undersampling and oversampling as they are simple to implement
while still allowing different batch sizes for the different terms.
Specifically, we first select which dataset we are interested in sampling fully, and then either
undersample or oversample the other data. 
For oversampling, we sample circularly, starting from the first data-points when we have passed through all data.

\section{Numerical convergence of multiplicative updates}
\label{app:convergence}
We run a convergence experiment to see if the proposed numerical algorithms behave as expected.
We synthetically generate $N = 2500$ strong supervised data of mixed "zero" and "one" digits, and use sparsity parameters $\lambda = \gamma = 10^{-10}$. For MDNMF and D+MDNMF we use $\tau_W = 1$ and $\tau_A = 0.1$,
and for D+ANMF we use $\tau_S = 0.5$. For MDNMF we use both data from other sources and mixed data as adversarial data, and for D+MDNMF we use the
strong supervised data for both the strong supervised and the weak supervised term.

We only report the convergence for the basis and latent variables for the "zero" digit source or all methods, including DNMF and D+MDNMF which are fitted for all other sources at the same time.
The results are shown in Figure \ref{fig:conv}.

\begin{figure*}[!htb]
    \centering
    \includegraphics[width = 0.49\textwidth]{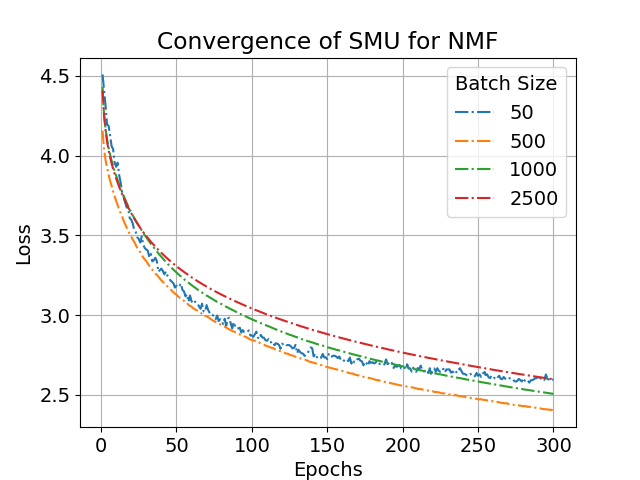}
    \includegraphics[width = 0.49\textwidth]{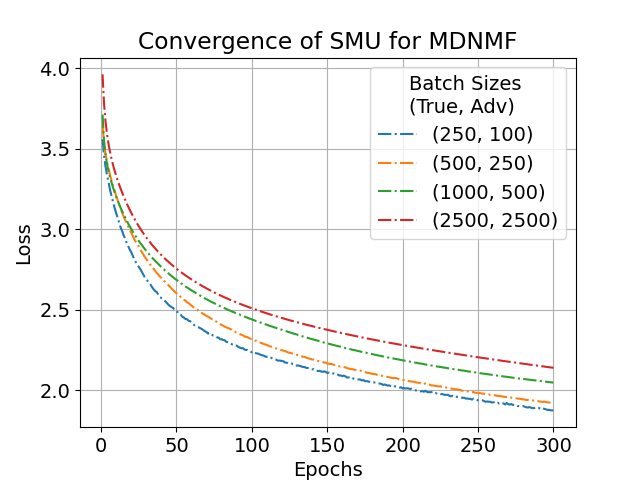}
    \includegraphics[width = 0.49\textwidth]{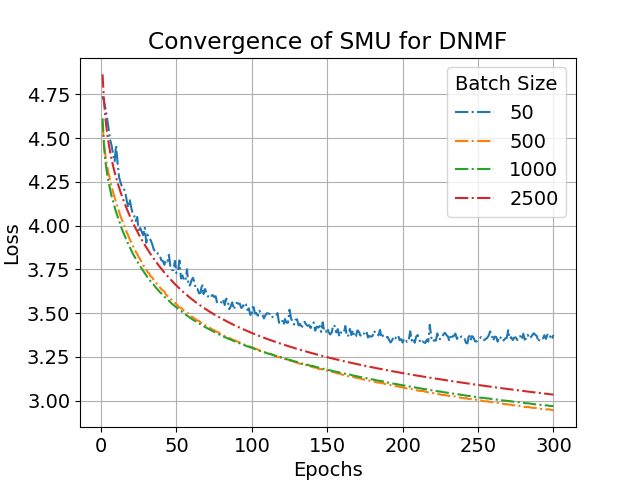}
    \includegraphics[width = 0.49\textwidth]{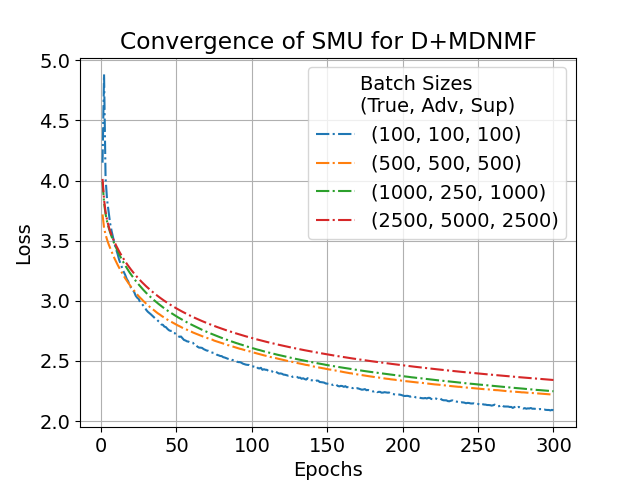}
    \caption{Convergence experiment on synthetic mixed data for different methods. We see that for the deterministic algorithms, the loss is non-increasing.
    Careful selection of batch sizes can lead to improved convergence speed, as well as faster computation time in some cases. 
    Selecting the batch size too low can hurt convergence speed. We note that D+MDNMF uses randomized initialization instead of exemplar based initialization.}
    \label{fig:conv}
\end{figure*}

\end{document}